\documentclass[12pt]{article}
\usepackage{amsmath,amssymb,amsthm,graphicx,enumerate,natbib,a4wide}

\newtheorem{theo}{Theorem}
\newtheorem{lem}[theo]{Lemma}
\newtheorem{coro}[theo]{Corollary}
\newtheorem{prop}[theo]{Proposition}
\newtheorem{hypo}{Assumption}
\theoremstyle{remark}
\newtheorem{rem}{Remark}

\begin{document}

\title{Estimation of conditional laws given an extreme component}

\author{Anne-Laure Foug\`eres\thanks{Universit\'e Claude Bernard - Lyon 1} \and
  Philippe Soulier\thanks{Universit\'e Paris Ouest-Nanterre}}

\maketitle

 \begin{abstract}
   Let $(X,Y)$ be a bivariate random vector. The estimation of a
   probability of the form $P(Y\leq y \mid X >t) $ is challenging when
   $t$ is large, and a fruitful approach consists in studying, if it
   exists, the limiting conditional distribution of the random vector
   $(X,Y)$, suitably normalized, given that $X$ is large.  There
   already exists a wide literature on bivariate models for which this
   limiting distribution exists. In this paper, a statistical analysis
   of this problem is done.  Estimators of the limiting distribution
   (which is assumed to exist) and the normalizing functions are
   provided, as well as an estimator of the conditional quantile
   function when the conditioning event is extreme. Consistency of the
   estimators is proved and a functional central limit theorem for the
   estimator of the limiting distribution is obtained. The small
   sample behavior of the estimator of the conditional quantile
   function is illustrated through simulations. Some real data are analysed.
 \end{abstract}

\section{Introduction}
\label{sec:introduction}

Let $(X,Y)$ be a bivariate random vector for which the conditional distribution
of $Y$ given that $X>t$ is of interest, for values of $t$ such that the
conditioning event is a rare event. This happens for example when the possible
contagion between two dependent market returns $X$ and $Y$ is investigated, see
e.g.  \cite{bradley:taqqu:2004} or \cite{abdous:fougeres:ghoudi:soulier:2008}.
The estimation of a probability of the form $P(Y\leq y \mid X >t) $ starts to
be challenging as soon as $t$ is large, since the conditional empirical
distribution becomes useless when no observations are available. A fruitful
alternative approach consists in studying, if it exists, the limiting
distribution of the random vector $(X,Y)$ conditionally on $X$ being large.
This corresponds to assuming that there exist functions $m$, $a$ and $\psi$,
and a bivariate distribution function (cdf) $F$ on
$[0,\infty)\times(-\infty,\infty)$ with non degenerate marginal distributions,
such that
\begin{gather}
  \lim_{t\to\infty} \mathbb{P}(X \leq t + \psi(t) x \; ; \; Y \leq m(t) + a(t) y \mid X >
  t) = F(x,y)  \label{eq:loi-limite} 
\end{gather}
at all points of continuity of $F$. This assumption has been called a
``conditional extreme value'' (CEV) model by \cite{das:resnick:2009}.  Some of
its main consequences were thoroughly investigated by
\cite{heffernan:resnick:2007} and its relationship to standard bivariate
extreme value theory has been investigated in \cite{das:resnick:2008}.  In
particular, Condition~(\ref{eq:loi-limite}) does not imply that the
distribution of $(X,Y)$ belongs to the domain of attraction of a bivariate
extreme value distribution. Examples of random vectors which are not in the
domain of attraction of a bivariate extreme value distribution and still
satisfy Condition~(\ref{eq:loi-limite}) are given in \cite{das:resnick:2008}
and \cite{fougeres:soulier:2010} (see Section 3.1 therein for a related
discussion).

The classical bivariate extreme value condition means that there exist
normalizing functions $c_1,c_2$, $d_1,d_2$ and a bivariate extreme value
distribution $H$ such that 
\begin{align} \label{eq:EV} \lim_{n\to\infty}\mathbb{P}^n(X \leq c_1(n)+d_1(n)x \; , Y
  \leq c_2(n)+d_2(n)y) =  H(x,y) \; .
\end{align}
The function $c_1$ can always be chosen such that $\lim_{s\to\infty}
s\mathbb{P}(X>c_1(s)) = 1$, and then it is easily seen that 
\begin{align*}
  \lim_{s\to\infty} \mathbb{P}(X \leq c_1(s)+d_1(s)x \, , \, Y \leq  c_2(s)+d_2(s)y \mid
  X>c_1(s)) = - \log H(0,y) + \log H(x,y) \; .
\end{align*}
Two cases are possible: either the limiting distribution $H$ is in product form
or it is not in product form. In extreme value theory, the former case is
referred to as asymptotic independence, and the latter is referred to as
asymptotic dependence.  In the case of asymptotic dependence, then $F(x,y) = -
\log H(0,y) + \log H(x,y)$ is a non degenerate distribution function, and the
CEV condition~(\ref{eq:loi-limite}) holds. Thus the problem is entirely solved
by the standard extreme value theory in the case of asymptotic dependence. In
the case of asymptotic independence, i.e. $H(x,y) = H_1(x)H_2(y)$, then $- \log
H(0,y) + \log H(x,y) = -\log H_1(x)$, so that the limit is degenerate with
respect to $y$. In the case of asymptotic independence, standard bivariate
extreme value theory is useless to check the CEV
condition~(\ref{eq:loi-limite}).  To summarize this discussion, it appears that
the CEV model~(\ref{eq:loi-limite}) is potentially useful when the standard
bivariate extreme value condition~(\ref{eq:EV}) does not hold or when
condition~(\ref{eq:EV}) holds with asymptotic independence.

It must be noted that the CEV condition~(\ref{eq:loi-limite}) always holds when
$X,Y$ are independent, $Y$ is non-degenerate and $X$ is in the domain of
attraction of a univariate extreme value distribution. This implies that all
bivariate distributions $F$ that can be expressed as $F(x,y)=F_1(x) F_2(y)$,
where $F_1$ is a univariate extreme value distribution and $F_2$ is \emph{any}
non-degenerate distribution function can appear as a limiting distribution
in~(\ref{eq:loi-limite}). This trivial consideration has the consequence that
it is very difficult to have a general theory for the conditional extreme value
model and to define statistical procedures valid in all cases. Rather, it is
necessary to study classes of bivariate distributions that satisfy the CEV
condition~(\ref{eq:loi-limite}) and to establish statistical procedures suited
to these classes.

Models for which condition~(\ref{eq:loi-limite}) holds have already been
investigated in many references.  \cite{eddy:gale:1981} and
\cite{berman:1982,berman:1992} proved that~(\ref{eq:loi-limite}) holds for
spherical distributions; bivariate elliptical distributions were investigated
by \cite{berman:1982}, multivariate elliptical distributions and related
distributions by \cite{hashorva:2006,hashorva:2007asymp}
\cite{hashorva:kotz:kume:2007}. The analysis of the underlying geometric
structure (ellipticity of the level sets of the densities) has lead to various
generalizations by \cite{barbe:2003} and \cite{balkema:embrechts:2007}.  See
also \cite{fougeres:soulier:2010} for a recent review on the subject. An
important finding of these works is that when the CEV condition is a
consequence of such geometric properties, the domain of attraction of the
conditioning variable determines the nature of the limiting distribution $F$
in~(\ref{eq:loi-limite}). For instance, for the usual bivariate elliptical
distributions, if the conditioning variable has a regularly varying right tail,
then the limiting distribution $F$ is not in product form, whereas it is in
product form if the conditioning variable is in the domain of attraction of the
Gumbel law.  Other types of models that satisfy
Assumption~(\ref{eq:loi-limite}) have been studied in
\cite{alink:lowe:wuthrich:2004},
\cite{hashorva:2008beta,hashorva:2009polar,hashorva:2009wp}.

The aim of this paper is the statistical estimation of the functions $a$ and
$m$ that appear in (\ref{eq:loi-limite}), as well as the limiting distribution
function $F$.  Two problems are considered. The first one is the nonparametric
estimation of the limiting distribution and of the normalizing functions.  This
is done in Section~\ref{sec:NPestim} in full generality. We only assume that
the CEV condition~(\ref{eq:loi-limite}) holds and some additional moment
conditions which are necessary to obtain consistency of the estimators. In
order to obtain central limit theorems, we also need to assume some kind of
second order conditions.

As discussed above, it seems impossible to go beyond these results without
making some restrictive assumptions.  From Section~\ref{sec:product} onwards,
we assume that the conditioning variable is in the maximum domain of attraction
of the Gumbel law, and that the limiting joint distribution $F$ has product
form. This choice leaves out many interesting cases, but is motivated by
previous works and has not been considered yet in a statistical study.  Under
these assumptions, we validate a Kolmogorov-Smirnov type test for the limiting
distribution of the $Y$ variable, e.g. the standard Gaussian distribution which
appears in many examples.  Since we are also interested in the case where the
conditioning event is beyond the range of observations, a semiparametric
procedure is defined to allow this extrapolation. This again necessitates
restrictive assumptions. Those we make are satisfied by several models already
investigated (cf. \cite{fougeres:soulier:2010}).  Let us finally note that to
the best of our knowledge, the estimators of the quantities related to the
conditional laws presented in this paper have not been considered before.

The paper is organized as follows.  In Section~\ref{sec:assump-prelim}, we
rephrase (\ref{eq:loi-limite}) in terms of vague convergence of measures in
order to use the point process techniques and the results of
\cite{heffernan:resnick:2007}.  We also introduce moment assumptions which are
needed to prove the consistency of the nonparametric estimators introduced in
Section~\ref{sec:NPestim}.  A functional central limit theorem for the
estimator of the limiting distribution is obtained under a second order
condition.  In Section~\ref{sec:product}, a specific analysis of the case of a
limiting distribution with product form and $X$ in the domain of attraction of
the Gumbel law is done.  The functional central limit theorem is used to derive
a goodness of fit test for the second marginal of the limiting distribution
$F$.  In Section~\ref{subsec:SP-estim}, semiparametric estimators that allow
extrapolations beyond the range of the observations are studied and applied to
the estimation of conditional quantiles when the conditioning event is extreme.
A simulation study is given in Section~\ref{sec:simulations} to illustrate the
small sample behavior of our estimators, of the goodness of fit test proposed
in Section~\ref{subsec:nonparam-estim} and of the estimator of the conditional
quantile proposed in Section~\ref{subsec:SP-estim}. These results are applied
in Section~\ref{sec:data} to some financial data.  Section~\ref{sec:preuves}
collects the proofs.

\section{Assumptions and preliminary results}
\label{sec:assump-prelim}
We first rephrase the convergence~(\ref{eq:loi-limite}) in terms of vague
convergence of measures, in order to use point process techniques and the
results of \cite{heffernan:resnick:2007}. See also
\cite{das:resnick:2009,das:resnick:2008}.  Condition~(\ref{eq:loi-limite})
implies that the marginal distribution of $X$ belongs to the domain of
attraction of an extreme value distribution with index $\gamma\in\mathbb R$,
i.e.  there exist normalizing sequences $\{b_n\}$ and $\{c_n\}$ with $c_n>0$
such that $\mathbb{P}(\max_{1 \leq i \leq n} (X_i-b_n)/c_n \leq x)$ converges
to $\exp\{-\bar P_\gamma(x)\}$ for each $x$ such that $1 + \gamma x > 0$, where
$\bar P_\gamma(x) = (1+\gamma x)^{-1/\gamma}$ if $\gamma\ne0$ and $\bar P_0(x)
= \mathrm e^{-x}$, and the random variables $X_i$ are independent copies of
$X$.
For simplicity, we assume that $\gamma\geq0$, and in the case
$\gamma=0$ we assume that the right endpoint of the marginal
distribution of $X$ is infinite.

Recall that a measure defined on the Borel sigma-field of a locally
compact separable space $E$ is called a Radon measure if it is finite
on compact sets.  A sequence of Radon measures $\sigma_n$ defined on
$E$ converges vaguely to a Radon measure $\sigma$ if $\int_E f(x)
\sigma_n(\mathrm d x)$ converges to $\int_E f(x) \sigma(\mathrm d x)$ for all
compactly supported function $f$. See \citet[Chapter~3]{resnick:1987}
or \citet[Appendix~A3]{heffernan:resnick:2007}. We will consider vague
convergence of Radon measures defined on the Borel sigma-fields of
$(-1/\gamma,\infty]$ or $(-1/\gamma,\infty]\times[-\infty,\infty]$.

\label{sec:hypo}
\begin{hypo}
  \label{hypo:exces}
  There exist $\gamma\geq0$, monotone functions $a$, $b$, $m$ and
  $\psi$ such that the marginal distribution of $X$ is in the domain
  of attraction of the extreme value distribution with extreme value
  index $\gamma$ and the sequence of measures $\nu_n$ defined by
  \begin{align*}
    \nu_n (\cdot) = n \mathbb{P} \left( \left\{\frac{X - b(n)}{\psi\circ
          b(n)}, \frac{Y - m \circ b(n)}{a\circ b(n)}\right\} \in
      \cdot \right)
  \end{align*}
  converges vaguely on $(-1/\gamma,\infty] \times
    [-\infty,\infty]$ to a Radon measure $\nu$ such that
  $\nu([0,\infty)\times(-\infty,\infty))=1$, the distribution function
  $y\mapsto\nu([0,\infty)\times(-\infty,y])$ is non degenerate and
    the application $(x,y)\mapsto \nu([x,\infty)\times(-\infty,y])$ is
    continuous on $(-1/\gamma,\infty] \times [-\infty,\infty]$.
\end{hypo}

Assumption~\ref{hypo:exces} is equivalent to assumptions 1.2 and 1.3 of
\cite{das:resnick:2009} and to Assumption (5) of \cite{heffernan:resnick:2007},
apart from the continuity assumption which is needed for statistical purposes
such as the Kolmogorov-Smirnov test proposed in
Section~\ref{subsec:nonparam-estim}.   The link between
Assumption~\ref{hypo:exces} and Equation~(\ref{eq:loi-limite}) is that the
limiting distribution $F$ is given, for all positive $x$ and real $y$, by
\begin{align*}
  F(x,y) = \nu([0,x] \times (-\infty,y]) \; .
\end{align*}
Assumption~\ref{hypo:exces} also implies that $F$ is continuous and
that the sequence of probability distribution functions $F_n$ defined,  for all positive $x$ and real $y$,
by
 \begin{align*}
   F_n(x,y) = \nu_n([0,x]\times(-\infty,y]) 
 \end{align*}
 converges to $F$ locally uniformly.  Assumption~\ref{hypo:exces} can
 also be interpreted as the weak convergence to $F$ of the vector
 $(X-b(n))/\psi\circ b(n),(Y-m\circ b(n))/a\circ b(n))$ conditionally
 on $X>b(n)$, i.e. for all bounded continuous function $h$ on
 $[0,\infty)\times (-\infty,\infty)$,
 \begin{gather}  \label{eq:weak-convergence}
   \lim_{n\to\infty} \mathbb{E} \left[ h\left(\frac{X-b(n)}{\psi\circ
         b(n)},\frac{Y-m\circ b(n)}{a\circ b(n)} \right) \mid X>b(n)
   \right] = \int_0^\infty \int_{-\infty}^\infty h(x,y) F(\mathrm d x,\mathrm d y)
   \; .
 \end{gather}
\begin{rem} \label{rem:unidim}
  All results concerning only the marginal distribution of $X$ are
  obtained by applying the usual extreme value theory. In particular,
  the functions $\psi$ and $b$ are determined by the marginal
  distribution of $X$ only. The function $b$ can and will be chosen as
  $b=(1/(1-F_X))^\leftarrow$ where $F_X$ is the distribution function
  of $X$.  The function $\psi$ satisfies
  \begin{gather}
    \lim_{x \to +\infty} \frac{\psi(x+\psi(x)u)}{\psi(x)} = 1 + \gamma u \;
    . \label{eq:propriete-psi}
  \end{gather}
  See \cite[Propositions 1.4 and 1.11]{resnick:1987}.  For any
  $x>-1/\gamma$, it holds that  
  \begin{gather*}
    \nu([x,\infty] \times [-\infty,\infty]) = (1+\gamma x)^{-1/\gamma}  \; , 
  \end{gather*}
  with the usual convention that this expression must be read as
  $\mathrm e^{-x}$ when $\gamma=0$.
  \end{rem}

  \begin{rem}
    Assumption~\ref{hypo:exces} has little implications on the
    functions $a$ and $m$ and on the distribution $\Psi$ defined by
  \begin{gather*}
    \Psi(z) = \int_0^\infty \int_{-\infty}^z \nu(\mathrm d x, \mathrm d y) \; .
  \end{gather*}
  If $Y$ is independent of $X$, then $\Psi$ is the distribution of $Y$,
  $a \equiv 1$ and $m\equiv0$.  Thus $\Psi$ can be {\em any} probability
  distribution and it is not necessarily an
  extreme value distribution.
\end{rem}

\begin{rem}
  If the pair $(X,Y)$ satisfies Assumption~\ref{hypo:exces}, then so does any
  affine coordinatewise transformation of $(X,Y)$.  For instance, if $X$ and
  $Y$ have finite mean and variance, then
  $((X-\mathbb{E}[X])/\mathrm{var}^{1/2}(X),(Y-\mathbb{E}[Y])/\mathrm{var}^{1/2}(Y))$
  also satisfies Assumption~\ref{hypo:exces}.  But non linear transformations
  of $(X,Y)$ do not necessarily satisfy the assumption. Even though the
  conditioning variable can be standardized, the simultaneous transformation of
  $X$ and $Y$ to random variables with prescribed marginal distributions is not
  always possible. This problem has been thoroughly investigated in
  \cite[Section~7]{heffernan:resnick:2007}. It is never possible in the cases
  where the joint limiting distribution is a product measure. Consequently, we
  do not make any specific assumption on the marginal distributions of $X$
  and~$Y$.

\end{rem}

  Obviously, the functions $a$ and $m$ are defined up to asymptotic
  equivalence, i.e. if $m'$ and $a'$ satisfy
  \begin{gather*}
    \lim_{x\to\infty} \frac{a'(x)}{a(x)} = 1 \; , \ \
    \lim_{x\to\infty} \frac{m(x) -m'(x)}{a(x)} = 0 \; ,
  \end{gather*}
  then the measure $\nu'_n$ defined as $\nu_n$ but with $a'$ and $m'$
  instead of $a$ and $m$ converges vaguely to the same limit measure
  $\nu$. Beyond this trivial remark, the following result summarizes
  \citet[Propositions~1 and~2]{heffernan:resnick:2007} and contains most
  of what can be infered from Assumption~\ref{hypo:exces}.  Recall
  that a function $f$ defined on a neighborhood of infinity is said to
  be regularly varying if there exists a constant $\alpha\in\mathbb R$
  such that
  \begin{align*}
    \lim_{x\to\infty} \frac{f(tx)}{f(x)} = t^\alpha 
  \end{align*}
  for all $t>0$. If $\alpha=0$, the function is called slowly varying.

\begin{lem}  \label{prop:h-resnick}
  Under Assumption~\ref{hypo:exces}, there exists
  $\zeta\in{\mathbb R}$ such that the function $a \circ b$ is regularly
  varying at infinity with index $\zeta$ and the function $m$
  satisfies
  \begin{gather*}
    \lim_{t \to \infty} \frac{m\circ b (tx) - m \circ b (t)}{a\circ
      b(t)} = J_\zeta(x) \; ,
  \end{gather*}
  with $J_\zeta(x) = (x^\zeta-1)/\zeta$ if $\zeta\ne0$ and $J_0(x) = c
  \log(x)$ for some $c\in{\mathbb R}$, and the convergence is locally
  uniform on $(0,\infty)$.
\end{lem}

For a sequence $(X_i,Y_i)$, $1 \leq i \leq n$, let $X_{(n:i)}$ denote
the $i$-th order statistic and $Y_{[n:i]}$ denote its concomitant,
i.e.  $X_{(n:1)}$,\dots,$X_{(n:n)}$ is the ordering of $X_1,\dots,X_n$
in increasing order, and $Y_{[n:i]}$ is the $Y$-variable corresponding
to $X_{(n:i)}$.

Recall that an intermediate sequence is a sequence of integers $k_n$
such that $\lim_{n\to\infty} k_n = \lim_{n\to\infty} n/k_n = \infty$.
In accordance with common use and for the clarity of notation, the
dependence on $n$ will be implicit in the sequel.

Define the random measure
\begin{gather} \label{eq:def-tilde-nu_n}
  \tilde \nu_n = \frac 1 k \sum_{i=1}^n \delta_{(\{X_i - b(n/k)\}/\psi
    \circ b(n/k), \{Y_i - m \circ b(n/k)\}/a \circ
    b(n/k))} \; .
  \end{gather}
  Applying \citet[Proposition~5.3]{resnick:1986} (see also
  \citet[Exercise~3.5.7]{resnick:1987}), we straightforwardly obtain
  the following result.
  \begin{prop}     \label{prop:conv-tildenun}
    If Assumption~\ref{hypo:exces} holds, then for any intermediate
    sequence $k$, $\tilde \nu_n$ converges weakly to $\nu$ locally
    uniformly on $(-1/\gamma,\infty] \times [-\infty,\infty]$.
  \end{prop}
  Consequently, $\tilde \nu_n([0,x]\times (-\infty,y])$ converges weakly
  locally uniformly to $F(x,y)$. But $\tilde \nu_n$ is not an estimator, since
  its definition involves the unknown functions $a$ and $m$.  In order to
  define estimators of these functions, and of the distribution function $F$,
  we will need to prove convergence of integrals of unbounded functions with
  respect to the random measure $\tilde \nu_n$.  Therefore we need to
  strengthen Assumption~\ref{hypo:exces}.

\begin{hypo}  \label{hypo:moments}
  There exists $p^*>0$, $q^*>0$ such that for any
  $\epsilon\in(0,1/\gamma)$,
  \begin{align}
    \lim_{n\to\infty} \int_{-\epsilon}^\infty \int_{-\infty}^\infty
    |x|^{p^*} |y|^{q^*} \nu_n(\mathrm d x,\mathrm d y) = \int_{-\epsilon}^\infty
    \int_{-\infty} ^\infty |x|^{p^*} |y|^{q^*} \nu(\mathrm d x,\mathrm d y) \; .
  \label{eq:convintegrale}
  \end{align}

\end{hypo}

Condition~(\ref{eq:convintegrale}) can be seen as a strengthening
of~(\ref{eq:loi-limite}) and~(\ref{eq:weak-convergence}) in order to
obtain the convergence of conditional moments.  Under Assumption~2,  for all $0< p \leq p^*$
and $0<q \leq q^*$, it holds that
  \begin{gather}
    \lim_{t\to\infty} \frac{ \mathbb{E}[ (X-t)^p |Y - m(t)|^q \mid
      X>t]}{\psi^p(t) a^q(t)} = \int_0^\infty
    \int_{-\infty}^\infty x^p |y|^q \nu(\mathrm d x,\mathrm d y) \; .
  \label{eq:moments-conditionnels}
\end{gather}

For the reason mentioned in Remark~\ref{rem:unidim},
Assumption~\ref{hypo:exces} implies the convergence~(\ref{eq:convintegrale})
with $q^*=0$ and any $p^* < 1/\gamma$. Conversely, 
Assumption~\ref{hypo:moments} implies $\gamma<1/p^*$.  In applications, it
will be assumed that $q^*\geq2$.  The function $a$ and the limiting measure
$\nu$ are defined up to a change of scale, thus, without loss of generality, we
assume henceforth that
  \begin{gather} \label{eq:standardisation-variance}
    \int_0^\infty \int_{-\infty}^\infty y^2 \, \nu(\mathrm d x, \mathrm d y) =
    \int_{-\infty}^\infty y^2 \Psi(\mathrm d y) = 1 \; .
  \end{gather}

  \begin{prop}  \label{prop:moments-tildenun}
    If Assumptions~\ref{hypo:exces} and~\ref{hypo:moments} hold, then
    for any intermediate sequence $k$ and any continuous function $g$
    such that $|g(x,y)| \leq C(|x|\vee1)^{p^*}(|y|\vee1)^{q^*}$, for any
    $\epsilon\in(0,1/\gamma)$,
  \begin{align}    \label{eq:conv-moments-tildenun}
    \int_{-\epsilon}^\infty \int_{-\infty}^\infty g(x,y) \tilde \nu_n(\mathrm d x, \mathrm d
    y) \to_P \int_{-\epsilon}^\infty \int_{-\infty}^\infty g(x,y) \nu(\mathrm d x,\mathrm d y)
    \; .
  \end{align}
  \end{prop}
  For historical interest, we can also mention the following consequence of
  Assumption~\ref{hypo:exces}. This result was obtained by \citet[Theorem
  6.1]{eddy:gale:1981} under the restrictive additional assumption of a
  spherical distribution. Related results can also be found in
  \cite{hashorva:2007asymp,hashorva:2008beta} and \cite{nagaraja:david:1994}.

  \begin{prop} 
    \label{prop:coco} 
    Under Assumption~\ref{hypo:exces}, $(\{X_{(n:n)}-b(n)\}/\psi\circ
    b(n),\{Y_{[n:n]} - m \circ b(n)\}/a\circ b(n))$ converges weakly to $F$.
  \end{prop}

  Since $\Psi$ is the second marginal of $F$, this result implies 
  that $\{Y_{[n:n]} - m \circ b(n)\}/a\circ b(n)$ converges weakly to
  $\Psi$. If $F$ is a product measure, (or equivalently if $\nu$ is a
  product measure), then $\{Y_{[n:n]} - m \circ b(n)\}/a\circ b(n)$ is
  asymptotically independent of $\{X_{(n:n)}-b(n)\}/\psi\circ b(n)$ in the
  usual sense that the limiting distribution is a product measure.

  Let us finally mention that \cite{davydov:egorov:2000} obtained functional
  limit theorems for sums of concomitants corresponding to a number $k$ of
  order statistics such that $k/n \nrightarrow 0$.  Therefore their problem
  differs from ours. Their assumptions on the joint distribution of the random
  pairs are much weaker than Assumption~\ref{hypo:exces}, but their results are
  of a very different nature and it does not seem possible to use them to
  derive Propositions~\ref{prop:conv-tildenun}-\ref{prop:moments-tildenun} for
  instance.

\section{Nonparametric estimation of  $\psi$, $a$, $m$ and $F$}
\label{sec:NPestim}
In this section, we introduce nonparametric estimators of the
functions $\psi$, $m$, $a$ and $F$ based on i.i.d. observations
$(X_1,Y_1),\dots, (X_n,Y_n)$ of a bivariate distribution which satisfies
Assumption~\ref{sec:hypo}.

\subsection{Definitions and consistency}
\label{sec:consistency}

In order to estimate nonparametrically the limiting distribution $F$,
we first need nonparametric estimators of the quantities
$\psi(X_{(n:n-k)})$, $m(X_{(n:n-k)})$ and $a(X_{(n:n-k)})$, with $k$
an intermediate sequence, i.e.  such that $k\to\infty$ and $k/n\to0$.
The estimation of $\psi(X_{(n:n-k)})$ is a well known estimation
issue, see e.g.  \citet[Section~4.2]{dehaan:ferreira:2006}. If the
extreme value index $\gamma$ of $X$ is less than~1, then $\psi$ can be
estimated as the mean residual life.  Let $\hat\gamma$ be a consistent
estimator of $\gamma$ (see e.g.
\citet[Chapter~3]{dehaan:ferreira:2006} or
\citet[Chapter~5]{beirlant:goegebeur:teugels:segers:2004}) and define
\begin{gather}\label{eq:NP-psi}
  \hat\psi(X_{(n:n-k)}) = \frac{1-\hat\gamma}k \sum_{i=1}^k \{
  X_{(n:n-i+1)} - X_{(n:n-k)} \} \; .
\end{gather}
It follows straightforwardly from
Proposition~\ref{prop:moments-tildenun} that
$\hat\psi(X_{(n:n-k)})/\psi\circ b(n/k)\to_P1$.
If it is moreover assumed (as in Section~\ref{sec:product} below) that
$\gamma=0$, then the above estimator can be modified accordingly:
\begin{gather}\label{eq:NP-psi-gammazero}
  \hat\psi(X_{(n:n-k)}) = \frac 1 k \sum_{i=1}^k \{ X_{(n:n-i+1)} -
  X_{(n:n-k)} \} \; .
\end{gather}
In order to estimate $m$, define
\begin{gather}\label{eq:mhat-NP}
  \hat m(X_{(n:n-k)}) = \frac{\sum_{i=1}^k Y_{[n:n-i+1]}
    \{X_{(n:n-i+1)} - X_{(n:n-k)}\}} { \sum_{i=1}^k
    \{X_{(n:n-i+1)} - X_{(n:n-k)}\}} \; .
\end{gather}

\begin{prop} 
  \label{prop:estim-m}
  If Assumptions~\ref{hypo:exces} and~\ref{hypo:moments} hold with 
  $p^*\geq1$ and $q^*\geq1$, then, for any intermediate sequence $k$, it holds
  that
  \begin{align*}
    \frac{ \hat m(X_{(n:n-k)}) - m \circ b (n/k)}{a\circ b(n/k)} \to
    _P  \mu \; ,
  \end{align*}
  where $\mu = (1-\gamma) \int_0^\infty \int_{-\infty}^\infty xy \nu(\mathrm d x,\mathrm d y)$.
  If moreover $m(x) = \rho x$ and either $\mu=0$ and $a(x)=O(x)$ or
  $a(x) = o(x)$ then $\hat m(X_{(n:n-k)})/X_{(n:n-k)}$ is a consistent
  estimator of $\rho$.
\end{prop}

\begin{rem}
  A sufficient condition for $\mu=0$ is the symmetry of the measure $\nu$ with
  respect to the second variable. This happens e.g. when if $\nu$ is a product
  measure and the distribution $\Psi$ is symmetric.
\end{rem}

\begin{rem} 
  As explained above, the assumption $p^\star\geq1$ in
  Proposition~\ref{prop:estim-m} implies that $\gamma<1$. If $\gamma\geq1$,
  which implies that $|X|$ has an infinite mean, the previous estimators of
  $\psi$ and $m$ need not be consistent. Consistent estimators of $\gamma$ and
  $\psi$ can be found in \cite[Section 4.2]{dehaan:ferreira:2006} but it is not
  clear how to define a consistent estimator of $m$ in this context.
\end{rem}

We now estimate $a(X_{(n:n-k)})$. Many estimators can be defined, each
needing an ad hoc moment assumption. The one we have chosen needs $q^*
\geq 2$ in Assumption~\ref{hypo:moments}.  Define
\begin{gather} \label{eq:def-hata}
  \hat a(X_{(n:n-k)}) = \left\{\frac1k \sum_{i=1}^k \{Y_{[n:n-i+1]} -
  \hat m(X_{(n:n-k)}) \}^2 \right\}^{1/2}\; .
\end{gather}

\begin{prop} 
   \label{prop:estim-a} 
   If Assumptions~\ref{hypo:exces} and Assumption~\ref{hypo:moments} hold with
   $p^*\geq1$ and $q^*\geq2$, and if $\mu=0$, then, for any intermediate
   sequence $k$, it holds that
  \begin{gather*}
    \hat a(X_{(n:n-k)})/ a \circ b(n/k) \to _P 1 \; .
  \end{gather*}
\end{prop}

\begin{rem}
  Under Assumption~\ref{hypo:exces} with $q^*\geq2$, the first moment of $\Psi$
  is finite and if $\mu\ne0$, then $\hat a(X_{(n:n-k)})/a\circ b(n/k) \to _P
  \tau$, with
  \begin{gather} \label{eq:def:tau}
    \tau^2 = 1 - 2 \mu \int_{-\infty}^\infty y \Psi(\mathrm d y) + \mu^2 \; .
  \end{gather}

\end{rem}

We can now consider the nonparametric estimator of the limiting joint
distribution $F$. Define
\begin{multline}
  \hat F(x,y) \\
  = \frac1k \sum_{i=1}^k \mathbf{1}_{\{X_{(n:n-i+1)} \leq X_{(n:n-k)} + \hat
    \psi(X_{(n:n-k)}) x \}} \times \mathbf{1}_{\{Y_{[n:n-i+1]} \leq \hat m(
    X_{(n:n-k)}) + \hat a(X_{(n:n-k)}) y\}} \; .
\end{multline}
Denote $u_n = {\hat\psi(X_{(n:n-k)})}/{\psi\circ b(n/k)}$ and 
\begin{gather} \label{eq:def-xin} \tilde x_n = \frac{X_{(n:n-k)} - b(n/k)}{\psi
    \circ b(n/k)} \; , \ \ v_n = \frac{\hat a(X_{(n:n-k)})}{a \circ b(n/k)} \;
  , \ \ \ \xi_n = \frac{\hat m(X_{(n:n-k)}) - m \circ b (n/k)}{a \circ b (n/k)}
  \; .
  \end{gather}
  Then
  \begin{gather*}
    \hat F(x,y) = \tilde \nu_n([\tilde x_n,\tilde x_n+ u_n x]\times
    (-\infty,\xi_n+v_n y]) \; .
  \end{gather*}
  Thus Propositions~\ref{prop:conv-tildenun},~\ref{prop:estim-m}
  and~\ref{prop:estim-a} easily yield the consistency of $\hat
  F(x,y)$, as stated in the following theorem.
  \begin{theo} \label{theo:estim-F} Under Assumptions~\ref{hypo:exces}
    and~\ref{hypo:moments} with $p^*\geq1$ and $q^*\geq2$, if $\mu=0$, then for
    any intermediate sequence $k$, $\hat F(x,y)$ converges weakly to $F(x,y)$.
\end{theo}

We can also define an estimator of the second marginal $\Psi$ of $F$.
Denote
\begin{align} \label{eq:def-hatpsi}
  \hat \Psi(y) & = \frac1k \sum_{i=1}^k \mathbf{1}_{\{Y_{[n:n-i+1]} \leq
    \hat m( X_{(n:n-k)}) + \hat a(X_{(n:n-k)}) y\}} \; .
\end{align}
Then, under the assumptions of Theorem~\ref{theo:estim-F}, $\hat
\Psi$ also converges to $\Psi$.  Note that if $\mu\ne0$, then
$\hat \Psi(z)$ converges weakly to $\Psi(\mu + \tau z)$, with
$\tau$ defined in~(\ref{eq:def:tau}).

\subsection{Central limit theorems}
\label{subsec:clt}
In order to obtain central limit theorems, we need to strengthen
Assumptions~\ref{hypo:exces} and~\ref{hypo:moments}.

\begin{hypo} \label{hypo:second-ordre-modifie} 
  There exist positive real numbers $p^\dag$ and $q^\dag$, a function
  $c$ such that $\lim_{t\to\infty} c(t) = 0$ and a Radon measure
  $\mu^\dag$ on $(-1/\gamma,\infty)\times(-\infty,\infty)$ such that
  for any $\epsilon\in(0,1/\gamma)$, and any measurable function $h$
  such that $|h(x,y)|\leq (|x|\vee1)^{p^\dag} (|y|\vee1)^{q^\dag}$, it
  holds that
  \begin{align*}
    \int_{-\epsilon}^\infty \int_{-\infty}^\infty |h(x,y)| \mu^\dag(\mathrm d
    x,\mathrm d y) < \infty \; ,
  \end{align*}
and
  \begin{multline}
    \left|\int_{-\epsilon}^\infty \int_{-\infty}^\infty h(x,y)
      \nu_n(\mathrm d x, \mathrm d y) - \int_{-\epsilon}^\infty
      \int_{-\infty}^\infty h(x,y) \nu(\mathrm d x, \mathrm d y) \right| \\
    \leq c \circ b(n) \int_{-\epsilon}^\infty \int_{-\infty}^\infty
    |h(x,y)| \mu^\dag(\mathrm d x,\mathrm d y) \; . \label{eq:second-ordre-modifie}
  \end{multline}
\end{hypo}
\begin{rem} 
  Taking $h = \mathbf{1}_{[0,x]\times(-\infty,y]}$,
  (\ref{eq:second-ordre-modifie}) yields
  \begin{align} \label{eq:second-ordre-usuel} |F_n(x,y) - F(x,y)| \leq
    c\circ b(n) \mu^\dag([0,x]\times(-\infty,y])
\end{align}
where $F_n(x,y) = \nu_n([0,x]\times(-\infty,y])$. This is a classical second
order condition (see e.g.  \citet[Condition~4.1]{dehaan:resnick:1993}), which
gives a non uniform rate of convergence in Condition~(\ref{eq:loi-limite}). The
condition~(\ref{eq:second-ordre-modifie}) is stronger
than~(\ref{eq:second-ordre-usuel}) in the sense that it moreover gives a rate
of convergence for conditional moments. Since
Assumption~\ref{hypo:second-ordre-modifie} implies that the first marginal of
$F$ has finite moments up to the order $p^\dag$, it also implies that
$\gamma<1/p^\dag$.
\end{rem}
For a sequence $k$ depending on $n$, define the random measure
$\tilde\mu_n$ by
 \begin{gather*}
   \tilde \mu_n = k^{1/2} \left( \tilde\nu_n - \nu \right) \;
  \end{gather*}
  and denote
  \begin{gather*}
  W_n(x,y) = \tilde \mu_n((x,\infty)\times(-\infty,y]).
\end{gather*}
The next result states the functional convergence of $W_n$ in the
space $\mathcal D((-1/\gamma,\infty)\times(-\infty,\infty))$ of
right-continuous and left-limited functions, endowed with Skorohod's
$J_1$ topology.
\begin{prop} \label{prop:clt-fidi}
  If Assumption~\ref{hypo:second-ordre-modifie} holds with $p^\dag\geq2$ and
  $q^\dag\geq 4$ and if the sequence $k$ is chosen such that
  \begin{align} \label{eq:condition-second-ordre-k}
    \lim_{n\to \infty} k^{1/2} c \circ b(n/k) = 0 \;,
  \end{align}
  then $k$ is an intermediate sequence and the sequence of processes $
  W_n$ converges weakly in $\mathcal
  D((-1/\gamma,\infty)\times(-\infty,\infty))$ to a Gaussian process $W$ with
  covariance function
  \begin{gather}
    \mathrm{cov}(W(x,y),W(x',y')) = \nu([x\vee x',+\infty] \times
   [-\infty,y\wedge y']) \; . \label{eq:cov-W}
  \end{gather}
  Moreover, the sequence of random measures $\tilde \mu_n$ converges
  weakly (in the sense of finite dimensional distributions) to an
  independently scattered Gaussian random measure $W$ with control
  measure~$\nu$ on the space of measurable functions $g$ such that
  $|g(x,y)|^2 \leq C(x\vee1)^{p^\dag}(|y|\vee1)^{q^\dag}$, i.e.
  $W(g)$ is a centered Gaussian random variable with variance
\begin{gather*}
  \int_{-1/\gamma}^\infty \int_{-\infty}^\infty g^2(s,t) \, \nu(\mathrm d s , \mathrm d t )
\end{gather*}
and $W(g)$, $W(h)$ are independent if $\int gh \, \mathrm d \nu =0$.
\end{prop}
The proof is in section~\ref{sec:preuves}.  Applying
Proposition~\ref{prop:clt-fidi}, we easily obtain the following
corollary. For $i,j\geq0$, denote $g_{i,j}(x,y) =
x^iy^j\mathbf{1}_{\{x>0\}}$.
\begin{coro} \label{coro:tilde}
  Under the assumptions of Proposition~\ref{prop:clt-fidi} and if
  moreover $\mu= 0$, then
  \begin{align*}
    k^{1/2} \left\{ \frac{X_{(n:n-k)} - b(n/k)}{\psi \circ b(n/k)} \;,
      \frac{ \hat m(X_{(n:n-k)}) - m\circ b(n/k) }{a \circ b(n/k)} \;,
      \frac{\hat a(X_{(n:n-k)})}{a \circ b (n/k)} - 1 \right \}
  \end{align*}
  converges jointly  
  with $k^{1/2}(\tilde \nu_n-\nu)$ to a Gaussian vector which can be
  expressed as 
  $$
  (W(g_{0,0}), (1-\gamma) W(g_{1,1}), \frac12 W(g_{0,2})) \; .
  $$
\end{coro}

Proposition~\ref{prop:clt-fidi} and Corollary~\ref{coro:tilde}
straightforwardly yield a functional central limit theorem for the
estimator $\hat \Psi$ of $\Psi$ defined in~(\ref{eq:def-hatpsi}).
Recall that $F(x,y) = \nu([0,x] \times (-\infty,y])$. 
\begin{theo} \label{theo:fclt-empirique} If
  Assumption~\ref{hypo:second-ordre-modifie} holds with $p^\dag\geq2$
  and $q^\dag\geq 4$, if $\mu=0$, if $F$ (and hence $\Psi$) is
  differentiable and if the intermediate sequence $k$
  satisfies~(\ref{eq:condition-second-ordre-k}), then $k^{1/2} (\hat
  \Psi - \Psi)$ converges in $\mathcal{D}((-\infty,+\infty))$ to the
  process $M$ defined by
  \begin{gather} \label{eq:process-M}
    M(y) = W(0,y) - \frac{\partial F}{\partial x} (0,y) W(g_{0,0}) +
    \Psi'(y) \{(1-\gamma) W(g_{1,1}) + \frac12 W(g_{0,2}) y \} \; .
  \end{gather}
\end{theo}
We prove Theorem~\ref{theo:fclt-empirique} here in order to explain
the last two terms in the right hand side of~(\ref{eq:process-M}).
\begin{proof}[Proof of Theorem~\ref{theo:fclt-empirique}]
  Recall the definitions of $\tilde x_n$, $v_n$ and $\xi_n$
  in~(\ref{eq:def-xin}).  Then
\begin{align}
  k^{1/2} \{ \hat \Psi(y) - \Psi(y)\} & = k^{1/2} \{
  \tilde\nu_n([\tilde x_n,\infty) \times (-\infty,\xi_n + v_n y]) -
  \Psi(y) \} \nonumber  \\
  & = \tilde \mu_n([\tilde x_n,\infty)\times(-\infty, \xi_n + v_n y]) 
  \label{eq:terme-tildemun} \\
  & + k^{1/2} \{\nu( [\tilde x_n,\infty) \times (-\infty,\xi_n + v_n
  y]) - \Psi(y) \} \; . \label{eq:termes-extra}
 \end{align}
 By Proposition~\ref{prop:clt-fidi}, the term
 in~(\ref{eq:terme-tildemun}) converges weakly to $W(0,y)$.  By
 Corollary~\ref{coro:tilde} and the delta method, the term
 in~(\ref{eq:termes-extra}) converges weakly to
 $$
 - \frac{\partial F}{\partial x} (0,y) W(g_{0,0}) + \Psi'(y)
 \{(1-\gamma) W(g_{1,1}) + \frac12 W(g_{0,2}) y \} \; .
$$
\end{proof}

\section{Case of a product measure}
\label{sec:product} 
As explained in the introduction, to proceed further, we restrict the class of
models that we consider by assuming that the limiting measre $\nu$ has product
form and that the conditioning variable is in the domain of attraction of the
Gumbel law.  For examples of bivariate distributions that satisfy this
assumption see e.g.  \citet{fougeres:soulier:2010}.

\begin{hypo}   \label{hypo:varrap-produit} 
  The function $\psi$ is an auxiliary function satisfying
  $\lim_{x\to\infty}\psi(x)/x=0$, there exists $\rho\in{\mathbb R}$ such
  that $m(x) = \rho x$ and the measure $\nu$ is of the form
  \begin{align} \label{eq:nu-prod}
    \nu([x,\infty]\times(-\infty,y]) = \mathrm{e}^{-x} \Psi(y) \; ,
  \end{align}
  where $\Psi$ is a distribution function on ${\mathbb R}$.
\end{hypo}

A rank test of the assumption that $\nu$ is a product measure has been proposed
by \cite{das:resnick:2009}. The condition $\lim_{x\to\infty}\psi(x)/x=0$
implies that the extreme value index of $X$ is 0
(cf. \citet[Lemma~1.2]{resnick:1987}). Testing this assumption can be done
using standard likelihood ratio procedures; see e.g. \cite{hosking:1984}. The
assumption $m(x)=\rho x$ is satisfied by most known examples. Cf.
\cite{fougeres:soulier:2010} for a review of models satisfying these
assumptions.

We now recall the necessary and sufficient condition for $\nu$ to be a product
measure proved by \citet[Proposition~2]{heffernan:resnick:2007}.

\begin{lem} \label{lem:cns-product}
  The measure $\nu$ is a product measure if and only if $a \circ b$ is
  slowly varying at infinity and
  \begin{gather} \label{eq:nul}
    \lim_{t \to \infty} \frac{ b (tx) -  b (t)}{a\circ
      b(t)} = 0 \; .
  \end{gather}
\end{lem}
The main consequence of Assumption~\ref{hypo:varrap-produit} and of
Lemma~\ref{lem:cns-product} is that $\psi(x) = o(a(x))$ (by application of
\citet[Theorem B.2.21]{dehaan:ferreira:2006}) and this implies that given
$X>t$, $(X-t)/a(t)$ converges in probability to zero. We thus have the
following Corollary.
\begin{coro} \label{coro:random-centering} 
  If Assumptions~\ref{hypo:exces} and~\ref{hypo:varrap-produit} hold
  then, for all $x\geq0$ and $y\in(-\infty,\infty)$,
\begin{gather*}
  \lim_{t\to\infty} \mathbb{P}( X \leq t+\psi(t) x \, , \ Y - \rho X \leq a(t) y
  \mid X>t) = (1-\mathrm{e}^{-x}) \Psi(y) \; .
\end{gather*}
Define the measure $\nu_n^\ddag$ on $(-1/\gamma,+\infty) \times
[-\infty,+\infty]$ by
\begin{align} \label{eq:def-nu_n-ddag}
  \nu_n^\ddag(\cdot) = n \mathbb{P} \left( \left\{ \frac{X - b(n)}{\psi \circ
        b(n)}, \frac{Y - \rho X }{a \circ b(n)} \right\} \in
    \cdot\right) \; .
\end{align}
Then $\nu_n^\ddag$ converges vaguely on $(-\infty,+\infty] \times
[-\infty,+\infty]$ to~$\nu$.
\end{coro}

\subsection{Nonparametric estimation}\label{subsec:nonparam-estim}
Under Assumption~\ref{hypo:varrap-produit}, we can define
new estimators of $\rho$, $a$ and the marginal distribution $\Psi$
as follows:
\begin{gather}
  \hat \rho = \frac{\sum_{i=1}^k Y_{[n:n-i+1]} \{X_{(n:n-i+1)} -
    X_{(n:n-k)}\}} { \sum_{i=1}^k X_{(n:n-i+1)} \{X_{(n:n-i+1)} -
    X_{(n:n-k)}\}} \; ,
 \label{eq:NP-rho} \\
  \check a(X_{(n:n-k)}) = \left[\frac1k \sum_{i=1}^k \{Y_{[n:n-i+1]}
    - \hat \rho X_{(n:n-i+1)} \}^2 \right]^{1/2}\; ,
  \label{eq:NP-a}\\
  \check \Psi(z) = \frac1k \sum_{i=1}^k \mathbf{1}_{\{Y_{[n:n-i+1]} \leq
    \hat \rho X_{(n:n-i+1)} + \check a(X_{(n:n-k)}) z\}} \; .
    \label{eq:NP-Psi}
\end{gather}
\begin{theo} \label{theo:estim-produit}
  If Assumptions~\ref{hypo:exces},~\ref{hypo:moments} (with $p^*=1$
  and $q^*=2$) and~\ref{hypo:varrap-produit} hold and if $\mu=0$, then
  for any intermediate sequence $k$, $b(n/k)(\hat \rho-\rho)/a\circ
  b(n/k)$ converges weakly to 0, $\check a(X_{(n:n-k)})/a\circ b(n/k)$
  converges weakly to 1 and $\check \Psi$ is a consistent estimator of
  $\Psi$. If moreover $a(x)=o(x)$ then $\hat\rho$ converges weakly to
  $\rho$.

\end{theo}

The proof of Theorem~\ref{theo:estim-produit} is along the lines of
the proof of Propositions~\ref{prop:estim-m},~\ref{prop:estim-a} and
Theorem~\ref{theo:estim-F}.  The only difference is that instead of
the random measure $\tilde\nu_n$ defined in~(\ref{eq:def-tilde-nu_n})
we use the measure $\check \nu_n$, defined by
\begin{gather} \label{eq:def-checknu}
  \check\nu_n = \frac1k \sum_{i=1}^n \delta_{(\{X_i-b(n/k)\}/\psi\circ
    b(n/k), \{Y_i-\rho X_i\}/a\circ b(n/k))} \; , 
\end{gather}
which converges weakly to the measure $\nu$ for any intermediate
sequence $k$, as a consequence of
Corollary~\ref{coro:random-centering} and
\citet[Proposition~5.3]{resnick:1986}.  The details are omitted.

In order to prove central limit theorems, we now introduce a second
order assumption which is a modification of
Assumption~\ref{hypo:second-ordre-modifie} that accounts for the
random centering. Recall the measure $\nu_n^\ddag$ defined
in~(\ref{eq:def-nu_n-ddag}).

\begin{hypo} \label{hypo:second-ordre-produit-modifie} There exist
  positive real numbers $p^\ddag$ and $q^\ddag$, a function $\tilde c$
  such that $\lim_{t\to\infty} \tilde c(t) = 0$ and a Radon measure
  $\mu^\ddag$ on $(-1/\gamma,\infty)\times(-\infty,\infty)$ such that
  for any $\epsilon\in(0,1/\gamma)$, and any measurable function $h$
  such that $|h(x,y)|\leq (|x|\vee1)^{p^\ddag} (|y|\vee1)^{q^\ddag}$,
  it holds that
  \begin{align*}
    \int_{-\epsilon}^\infty \int_{-\infty}^\infty |h(x,y)|
    \mu^\ddag(\mathrm d x,\mathrm d y) < \infty \; ,
  \end{align*}
and
  \begin{multline}
    \left|\int_{-\epsilon}^\infty \int_{-\infty}^\infty h(x,y)
      \nu_n^\ddag(\mathrm d x, \mathrm d y) - \int_{-\epsilon}^\infty
      \int_{-\infty}^\infty h(x,y) \nu(\mathrm d x, \mathrm d y) \right| \\
    \leq \tilde c \circ b(n) \int_{-\epsilon}^\infty \int_{-\infty}^\infty
    |h(x,y)| \mu^\ddag(\mathrm d x,\mathrm d y) \; . \label{eq:second-ordre-produit-modifie}
  \end{multline}
\end{hypo}
The difference with Assumption~\ref{hypo:second-ordre-modifie} is the
presence of measure $\nu_n^\ddag$ instead of $\nu_n$. 
It can be shown that Assumptions~\ref{hypo:second-ordre-modifie}
and~\ref{hypo:varrap-produit} with a smoothness assumption on $\Psi$
imply Assumption~\ref{hypo:second-ordre-produit-modifie}, but with the
same rate function $c$ as in
Assumption~\ref{hypo:second-ordre-modifie}, whereas in some cases
Assumption~\ref{hypo:second-ordre-produit-modifie} can be proved
directly with a function $\tilde c$ which goes to zero at infinity
faster than $c$. The following results could be stated under
Assumption~\ref{hypo:second-ordre-modifie}, but the interest of
Assumption~\ref{hypo:second-ordre-produit-modifie} is to take into
account the possibility of faster rates of convergence of the
estimators than those allowed by
Assumption~\ref{hypo:second-ordre-modifie}.

As an example, consider the case of a bivariate Gaussian vector
with standard marginals and correlation~$\rho$.
\cite{abdous:fougeres:ghoudi:2005} have shown that $\lim_{x\to\infty}
\mathbb{P}(Y\leq \rho x + \sqrt{1-\rho^2}y \mid X>x) = \Phi(y)$ (where $\Phi$
is the distribution function of the standard Gaussian law), and a rate
of convergence of order $x^{-1}$ has been proved in
\cite{abdous:fougeres:ghoudi:soulier:2008}. But of course, since
$(Y-\rho X)/\sqrt{1-\rho^2}$ is standard Gaussian and independent of
$X$, for all $x$ it holds that $\mathbb{P}(Y\leq \rho X + \sqrt{1-\rho^2}y
\mid X>x) = \Phi(y)$. For general elliptical bivariate random vectors,
it is also proved in \cite{abdous:fougeres:ghoudi:soulier:2008} that
the rate of convergence with random centering can be the square of the
rate with deterministic centering.
Assumption~\ref{hypo:second-ordre-produit-modifie} can also be checked for the
generalized elliptical distributions studied
in~\cite{fougeres:soulier:2010}.

We can now state central limit theorems for $\check a(X_{(n:n-k)})$,
$\hat \rho$ and $\hat \Psi$ which parallels Corollary~\ref{coro:tilde}
and Theorem~\ref{theo:fclt-empirique}. The proof is also omitted.

\begin{theo} \label{theo:fclt-produit} If
  Assumptions~\ref{hypo:exces},~\ref{hypo:varrap-produit}
  and~\ref{hypo:second-ordre-produit-modifie} hold with $p^\ddag\geq2$
  and $q^\ddag\geq4$, if $\Psi$ is  differentiable and if
  $\mu=0$ and if the intermediate sequence $k$ is chosen such that
  \begin{gather} \label{eq:condition-second-ordre-k-modifie}
    \lim_{n\to\infty} k^{1/2} \tilde c \circ b(n/k) = 0 \; ,
  \end{gather}
  then $k^{1/2}\{\check \Psi - \Psi\}$ converges weakly in
  $\mathcal{D}((-\infty,\infty))$ to the process $M$ defined
  in~(\ref{eq:process-M}) and
  $$
  k^{1/2} \left( \frac{ b(n/k) (\hat\rho - \rho) }{ a \circ b(n/k)}
    \; , \frac{\check a(X_{(n:n-k)})}{a \circ b(n/k)} - 1 \right)
  $$
  converges jointly with $k^{1/2}(\check\Psi-\Psi)$ to the Gaussian
  vector $(W(g_{1,1}),W(g_{0,2}))$.
\end{theo}

\begin{rem}
  As mentioned above, if we only assume
  Assumption~\ref{hypo:second-ordre-modifie} instead of
  Assumption~\ref{hypo:second-ordre-produit-modifie}
  and~(\ref{eq:condition-second-ordre-k-modifie}) with $c$ instead of
  $\tilde c$ then the conclusion of the theorem still holds.
\end{rem}

\subsubsection*{Kolmogorov-Smirnov Test}
In the case $\gamma=0$ and when the limiting measure $\nu$ has product
form, then $\frac{\partial}{\partial x} F(0,y) = \Psi(y)$.  Define
$B(t) = W(0,\Psi^{-1}(t))$. Then $B$ is a standard Brownian motion on
$[0,1]$ and
$$
W(0,y) - \frac{\partial}{\partial x} F(0,y) W(g_{0,0}) = B\circ
\Psi(y) - \Psi(y) B(1) = \mathcal B \circ \Psi (y)
$$
where $\mathcal B$ is a standard Brownian bridge.  By the same
change of variable, $W(g_{0,2})$ can be represented as 
$$
V = \int_0^1 \{\Psi^{-1}(t)\}^2 \, \mathrm d B(t) \; .
$$
Since $\mu=0$ and $\int_{-\infty}^\infty y^2\Psi(\mathrm d y) = 1$, it is
easily seen that 
\begin{gather*}
  \mathrm{var} (W(g_{1,1})) = 2 \; , \ \ 
  \mathrm{cov}(W(g_{0,0}),W(g_{0,1})) = 0 \; ,  \\
  \mathrm{cov}(W(0,y), W(g_{1,1})) = \int_{-\infty}^y z \Psi(\mathrm d z) =
  \int_0^{\Psi(y)} \Psi^{-1}(u) \, \mathrm d u \; ,  \\
  \mathrm{cov}(W(g(1,1)), W(g_{0,2})) = \int_{-\infty}^\infty z^3 \Psi(\d
  z) = \int_0^\infty \{\Psi^{-1}(u)\}^3 \, \mathrm d u \; .
\end{gather*}
Thus, $W(g_{1,1})$ can be represented as 
\begin{gather*}
  U  = \int_0^1 \Psi^{-1}(s) \, \mathrm d  B(s) + N  \; , 
\end{gather*}
where $N$ is a standard Gaussian random variable independent of the
Brownian motion~$B$. Since all random variables involved are jointly
Gaussian, this shows that $M(y)$ has the same distribution as
\begin{align*}
  \mathcal B \circ \Psi(y) + \Psi'(y) \{ U + \frac{1}2 y V\} \; .
\end{align*}
Finally, since $\Psi$ is continuous, $\sup_{y\in{\mathbb R}} |M(y)|$
has the same distribution as
\begin{gather}\label{eq:proclim}
  Z = \sup_{t\in[0,1]} \left| \mathcal B(t) + \Psi' \circ \Psi^{-1}(t)
    \{ U + \frac{1}2 \Psi^{-1}(t) V\} \right| \; .
\end{gather}
The extra terms come from the estimation of the functions $a$ and $m$.
If they were known, the limiting distribution would be the Brownian
bridge as expected. Nevertheless, this distribution depends only on
$\Psi$, so it can be used for a goodness-of-fit test. See
Section~\ref{subsec:gof} for a numerical illustration.

\subsection{Semiparametric estimation} \label{subsec:SP-estim} 
Two problems arise in practice: the estimation of the conditional
probability $\theta(x,y) = \mathbb{P}(Y\leq y \mid X > x)$ and of the
conditional quantile $y=\theta^\leftarrow(x,p)$ for some fixed $p \in
(0,1)$ and for some extreme $x$, i.e. beyond the range of the
observations.

 If $x$ lies {\it within} the range of the observations, then
$\theta(x,y)$ can be estimated empirically by
\begin{align*}
  \hat\theta_{\mathrm{emp}}(x,y) = \frac1k \sum_{i=1}^n \mathbf{1}_{\{Y_i
    \leq y\}} \mathbf{1}_{\{X_i > x\}} \; ,
\end{align*}
for $x=X_{(n:n-k)}$.  The limit distributions that arise in
Assumption~\ref{hypo:exces} is very useful when $x$ is outside the range of the
observations, so that an empirical estimate is no longer available. In such a
situation, a semiparametric approach will be needed to extrapolate the
functions $a(x)$, $m(x)$ and $\psi(x)$ for values $x$ beyond $X_{(n:n)}$. This
requires some modeling restrictions.  We still assume that
Assumption~\ref{hypo:varrap-produit} holds and we assume moreover that there
exists $\sigma>0$ such that
\begin{gather} \label{eq:simplif}
  a(x) = \sigma\sqrt{x\psi(x)} \; . 
\end{gather}
We will also assume that the limiting distribution function $\Psi$ in
(\ref{eq:nu-prod}) is known. These assumptions hold in particular for
bivariate elliptical distribution, see
\cite{abdous:fougeres:ghoudi:soulier:2008}. There, and in many other
examples, $\Psi$ is the distribution function of the standard Gaussian
law.  See also \cite{fougeres:soulier:2010}.
Assumption~\ref{hypo:varrap-produit} and~(\ref{eq:simplif}) imply that
\begin{gather} \label{eq:approx-theta}
\lim_{x\to\infty} \theta(x,\rho x + \sigma \sqrt{x\psi(x)}z) =
\Psi(z) \, ,
\end{gather}
so that $\theta(x,y)$ can be approximated for $x$ large enough by
$$
\Psi\left(\frac{y-\rho x}{\sigma \sqrt{x\psi(x)}}\right) \; . 
$$
Thus, in order to estimate $\theta$, we need a semiparametric
estimator of $\psi$. For this purpose, we make the following
assumption on the marginal distribution of $X$.
\begin{hypo} \label{hypo:weibull}
  The distribution function $H$ of  $X$ satisfies
  \begin{gather*}
    1 - H(x) = \mathrm{e}^{-x^{\beta} \{c+O(x^{\beta\eta})\}} \; 
  \end{gather*}
  with $\beta>0$ and $\eta<0$.
\end{hypo}
Under Assumption~\ref{hypo:weibull}, an admissible auxiliary function
is given by
\begin{gather}
  \psi(x) = \frac1{c\beta}x^{1-\beta}  \; . \label{eq:psi-semiparametrique}
\end{gather}
Under~(\ref{eq:simplif}), the normalizing function $a$ is then
\begin{align*}
  a(x) = \frac\sigma{\sqrt{c\beta}}x^{1-\beta/2}  \; .
\end{align*}
Let $k$ and $k_1$ be intermediate sequences. For the sake of clarity, in the
sequel, we make explicit the dependence of the estimators with respect to $k$
or $k_1$.  Semiparametric estimation of $\beta$ has been widely investigated
recently, and pitfalls of the methods have been shown
by~\cite{asimit:deyuan:li:2010}. We consider here an estimator suggested
in~\cite{gardes:girard:2006}. Define
\begin{gather} \label{eq:hatbeta} \hat \beta_k = \frac{\sum_{i=1}^{k}
    \log\log(n/i) - \log\log(n/k)}{\sum_{i=1}^{k} \log(X_{(n:n-i+1)})
    -
    \log(X_{(n:n-k)})} \; .
\end{gather}
A semiparametric  estimator of $a$ is now defined by 
\begin{gather}
  \breve a_{k_1}(x) = \check a_{k_1} (X_{(n:n-k_1)}) \left(
    \frac{x}{X_{(n:n-k_1)}} \right)^{1-\hat \beta_k/2} \; ,
  \label{eq:def-breve-a}
\end{gather}
where $\check a_{k_1}(X_{(n:n-k_1)})$ is the nonparametric estimator
defined in~(\ref{eq:NP-a}).
\begin{prop} \label{prop:breve-a} If Assumption~\ref{hypo:weibull}
  holds, and if $k$ is an intermediate sequence such that
  \begin{gather} \label{eq:cond-k-estim-beta}
    \lim_{n\to\infty} \log(k)/\log(n) = \lim_{n\to\infty} k
    \log^{2\eta}(n) = 0 \; ,
  \end{gather} 
  then $k^{1/2}(\hat \beta_k-\beta)$ converges weakly to the centered Gaussian
  distribution with variance $\beta^{-2}$.  Suppose moreover that
  Assumptions~\ref{hypo:exces},~\ref{hypo:varrap-produit}
  and~\ref{hypo:second-ordre-produit-modifie} hold with $p^\ddag=2$ and
  $q^\ddag=4$ and that $\mu=0$ and~(\ref{eq:simplif}) holds. Let $(x_n)$ be a
  sequence and $k_1$ be an intermediate sequence such that
  \begin{gather}
    \lim_{n\to\infty } k_1^{1/2} \tilde c\circ b(n/k_1) = 0
    \label{eq:rate-k_1}    \\
    \lim_{n\to\infty}    k/k_1 = 0 \; ,  \label{eq:k/k_1to0} \\
    \lim_{n\to\infty} \log(b(n/k_1))/\log(x_n) = 1 \; ,
    \label{eq:logbnsimlogxn}    \\
    \lim_{n\to\infty} k^{-1/2}\log(x_n) = 0 \; . \label{eq:klogn}
\end{gather}
Then $\displaystyle \frac{k^{1/2}}{\log(x_n)}\left\{ \frac{\breve
    a_{k_1}(x_n)}{a(x_n)} - 1 \right\}$ converges weakly to the
centered Gaussian distribution with variance $\beta^{-2}$.
\end{prop}
\begin{rem}
  By the arguments following
  Assumption~\ref{hypo:second-ordre-produit-modifie}, it can be seen
  that the conclusion of Proposition~\ref{prop:breve-a} still holds if
  Assumption~\ref{hypo:second-ordre-produit-modifie} is replaced by
  Assumption~\ref{hypo:second-ordre-modifie} and $\tilde c$ is
  replaced by $c$ in~(\ref{eq:rate-k_1}).
\end{rem}

The previous results lead to  natural estimators of the conditional
probability $\theta(x,y) = \mathbb{P}(Y\leq y \mid X > x)$ and of the
conditional quantile $y=\theta^\leftarrow(x,p)$.  Define
\begin{gather} \label{eq:thetahat} 
  \hat \theta(x,y)= \Psi\left( \frac{y-\hat \rho x}{\breve a_{k_1}(x)}
  \right) \; .
\end{gather}
Under
Assumptions~\ref{hypo:exces},~\ref{hypo:moments},~\ref{hypo:varrap-produit}
and~(\ref{eq:simplif}), Theorem~\ref{theo:estim-produit} implies that
for {\it fixed} $x$ and $y$, $\hat \theta(x,y)$ is a consistent
estimator of $\Psi \left( (y-\rho x) / a(x) \right)$, but a biased
estimator of $\theta(x,y)$. The remaining bias, which is an
approximation error due to the asymptotic nature of equation
(\ref{eq:approx-theta}), can be bounded thanks to the second order
Assumption~\ref{hypo:second-ordre-produit-modifie}. For more details,
see \citet[Section 3.2]{abdous:fougeres:ghoudi:soulier:2008} for a
treatment in the elliptical case.

We now investigate more thoroughly the estimation of the conditional
quantile $y_n=\theta^\leftarrow(x_n,p)$ for some fixed $p \in (0,1)$
and some extreme sequence $x_n$, i.e. beyond the range of the
observations, or equivalently, $x_n >b(n)$.  An estimator $\hat y_n$
is defined by
\begin{gather}\label{eq:ynhat}
  \hat y_n = \hat \rho_{k_1} x_n + \breve a_{k_1}(x_n) \Psi^{-1}(p)
  \; ,
\end{gather} 
where $\hat\rho_{k_1}$ is the nonparametric estimator defined
in~(\ref{eq:NP-rho}).

\begin{coro} \label{coro:quantile} Let the assumptions of
  Proposition~\ref{prop:breve-a} hold with
  Assumption~\ref{hypo:second-ordre-modifie} instead of
  Assumption~\ref{hypo:second-ordre-produit-modifie} and $c$ instead of $\tilde
  c$ in~(\ref{eq:rate-k_1}), $\Psi'\circ\Psi^{-1}(p)>0$ and
  \begin{align*}
    \lim_{n\to\infty} \frac{b(n/k_1)}{b(n)} = \lim_{n\to\infty}
    \frac{b(n/k_1)}{x_n} = 1 \; .
  \end{align*}
  \begin{enumerate}[(i)]
  \item If $\Psi^{-1}(p)\ne0$, then 
  $$
  \frac{k^{1/2}x_n}{\log(x_n)a(x_n)} \left\{ \frac{\hat y_n}{y_n} - 1
  \right\}
  $$
  converges weakly to a centered Gaussian law with variance $\left\{ {
      \Psi^{-1}(p)}/{\rho\beta} \right\}^2$.
\item If $\Psi^{-1}(p)=0$, then 
  $$
  \frac{k_1^{1/2} x_n}{a(x_n)} \left\{ \frac{\hat y_n}{y_n} - 1
  \right\}
  $$
  converges weakly to a centered Gaussian law with variance $2$.
\end{enumerate}
\end{coro}

\section{Numerical Illustration}
\label{sec:simulations}
In this section, we perform a small sample simulation study with three
purposes. We first illustrate the small sample behavior of the nonparametric
estimator of $a$, $m$ and $\Psi$ in the general framework of
Section~\ref{sec:consistency}.  Then we restrict our study to the framework of
section~\ref{sec:product} where we assume that the limiting distribution is a
product and the extreme value index of the distribution of $X$ is zero.  We
analyze the behavior of the Kolmogorov-Smirnov test proposed in
Section~\ref{subsec:nonparam-estim} and we illustrate the behavior of the
estimator of the conditional quantile proposed in
Section~\ref{subsec:SP-estim}.

\subsection{Nonparametric estimation in the general case} 
\label{sec:simul-general}
We consider the bivariate distribution $C(G(x),G(y))$ where $C$ is Frank's
copula, defined for $u,v \in [0,1]$ by
$$
C(u,v) = \frac{1} {\log \theta} \log \left\{ 1- \frac{(1-\theta^u)
    (1-\theta^v)} {1-\theta} \right\} \; ,
$$
$\theta \in (0,1)$ and $G(x)=e^{-x^{-1/\gamma}}$ for $x\geq0$ and $\gamma<1/2$.
Assumption 2 is then fulfilled  with $p^* = 1$ and $q^* = 2$.  This distribution is
in the domain of attraction of a max-stable law with independent margins
(i.e. asymptotically independent in the sense of extreme value distribution).
As argued in the introduction, this is the case where the CEV model is most
interesting, since it provides additional information compared to the classical
extreme value model.
The limiting measure $\nu$ is given for $x>-1/\gamma$ and $y\in\mathbb R$ by
$$
\nu\{(x,\infty] \times (-\infty,y]\} = (1+\gamma x)^{-1/\gamma} \;
\frac{\theta^{1-G(y)}-\theta}{1-\theta} \; .
$$ 
Theoretical values of the normalizing functions are respectively $\psi(t)=
\gamma t$, $m(t)=0$ and $a(t)=1$, and the second margin of the limiting
distribution function is
$$
\Psi(y)=\frac{\theta^{1-G(y)}-\theta}{1-\theta} \; .
$$ 
Note that in this case $\mu \neq 0$, so that $\hat \Psi(z)$ defined by
(\ref{eq:def-hatpsi}) converges weakly to $\Psi(\mu + \tau z)$, where $\tau^2 =
\int y^2 \Psi(dy) - 2 \mu \int y \Psi(dy) + \mu^2$.

The shape parameter $\gamma$ is estimated via the Hill estimator (see
e.g. \citet[Chapter~4]{beirlant:goegebeur:teugels:segers:2004}), which requires
specifying how many upper order statistics will be used. This first threshold
will be denoted by $k_\gamma$ in the following. Using the estimates of $\psi$,
$m$, $a$ and $\Psi$ respectively defined by (\ref{eq:NP-psi}),
(\ref{eq:mhat-NP}), (\ref{eq:def-hata}) and (\ref{eq:def-hatpsi}) also requires
choosing a second threshold, which is the number of observations kept with
largest first component, denoted by $k$ in Section~3.  We have compared the
theoretical value of $\Psi(\mu + \tau \, \cdot )$ with the estimate given by
$\hat \Psi$ for different values of $\theta \in (0,1)$ and $\gamma \in
(0,1/2)$, sample size $n=10^4$ and different thresholds ($k=200, 500, 1000,
2000, 3000$ and $k_\gamma=50, 500$).  Figure~\ref{fig:psiNPplot} illustrates
the estimation of $ \Psi$ via the nonparametric estimator $\hat \Psi$ defined
by (\ref{eq:def-hatpsi}) for one sample.
\begin{figure}[htbp]
 \centering 
 \includegraphics[width=12cm,height=5cm]{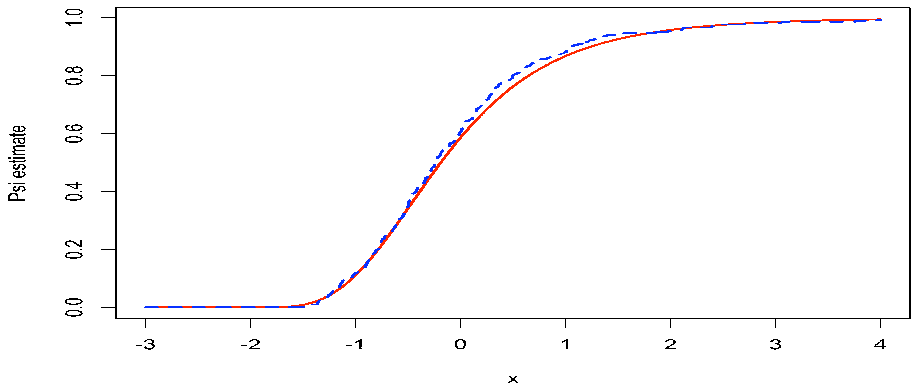}
 \caption{ Estimation of $\Psi(\mu + \tau \, \cdot ) $ via the nonparametric
   estimator $\hat \Psi$ for one sample ($n = 1000, k= 300, k_\gamma = 50$) of
   distribution with Fr\'echet margins of parameter $1/\gamma$ and Frank's
   copula with parameter $\theta$, for $\gamma=0.1$ and $\theta=0.3$.}
     \label{fig:psiNPplot}
\end{figure}
In each case listed above, 100 samples have been simulated, and for each of
them, the $L^1$-distance $\int |\hat \Psi(z) - \Psi(\mu + \tau z) | dz$ has
been calculated.  A summary of the results obtained is provided in Figure
\ref{fig:distL1_10e4_th03_gam01}, which give the boxplots of the
$L^1$-distances for $\theta=0.3$, $\gamma=0.1$ or $\gamma=0.4$, and
$n=10^4$. The results obtained for other values of $\theta$ were very similar,
so for brevity we do not present them.

\begin{figure}[h!]
\begin{tabular}{cc}
\includegraphics[width=7.5cm,height=8cm]{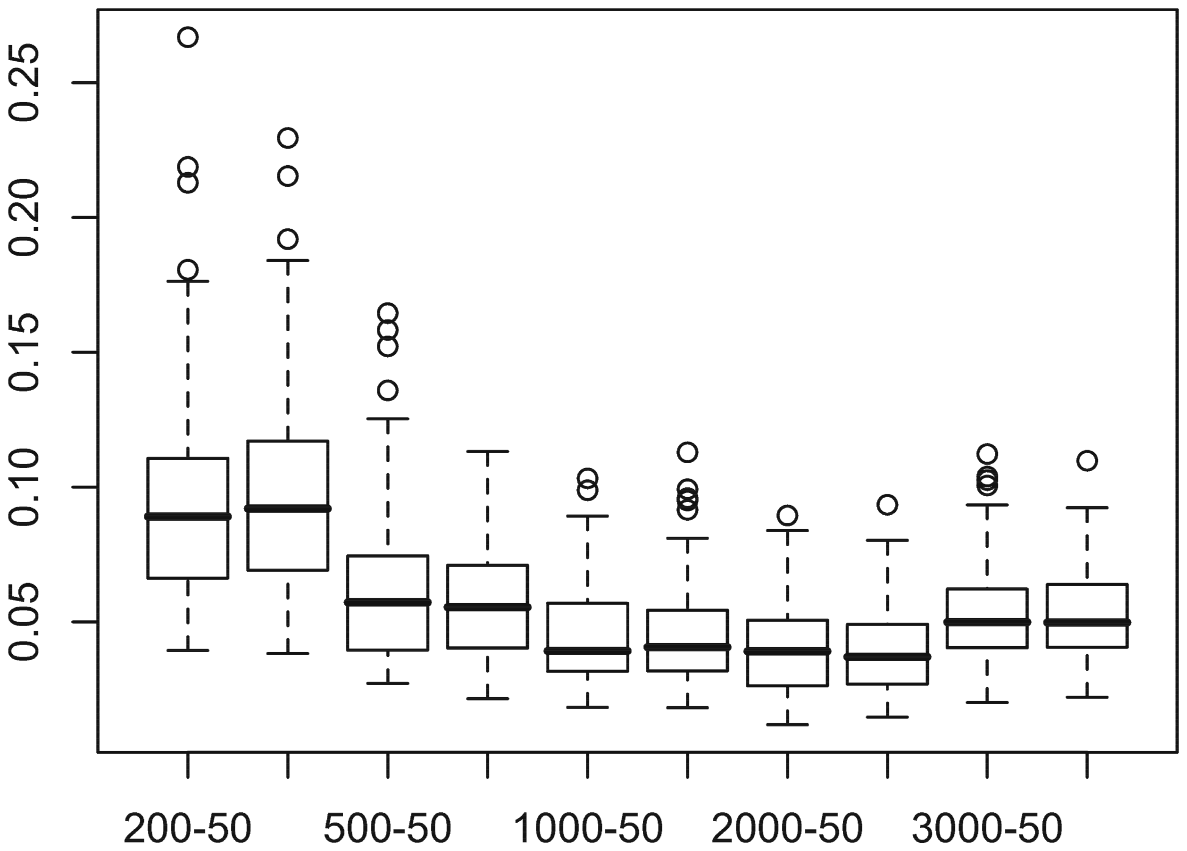}&
\includegraphics[width=7.5cm,height=8cm]{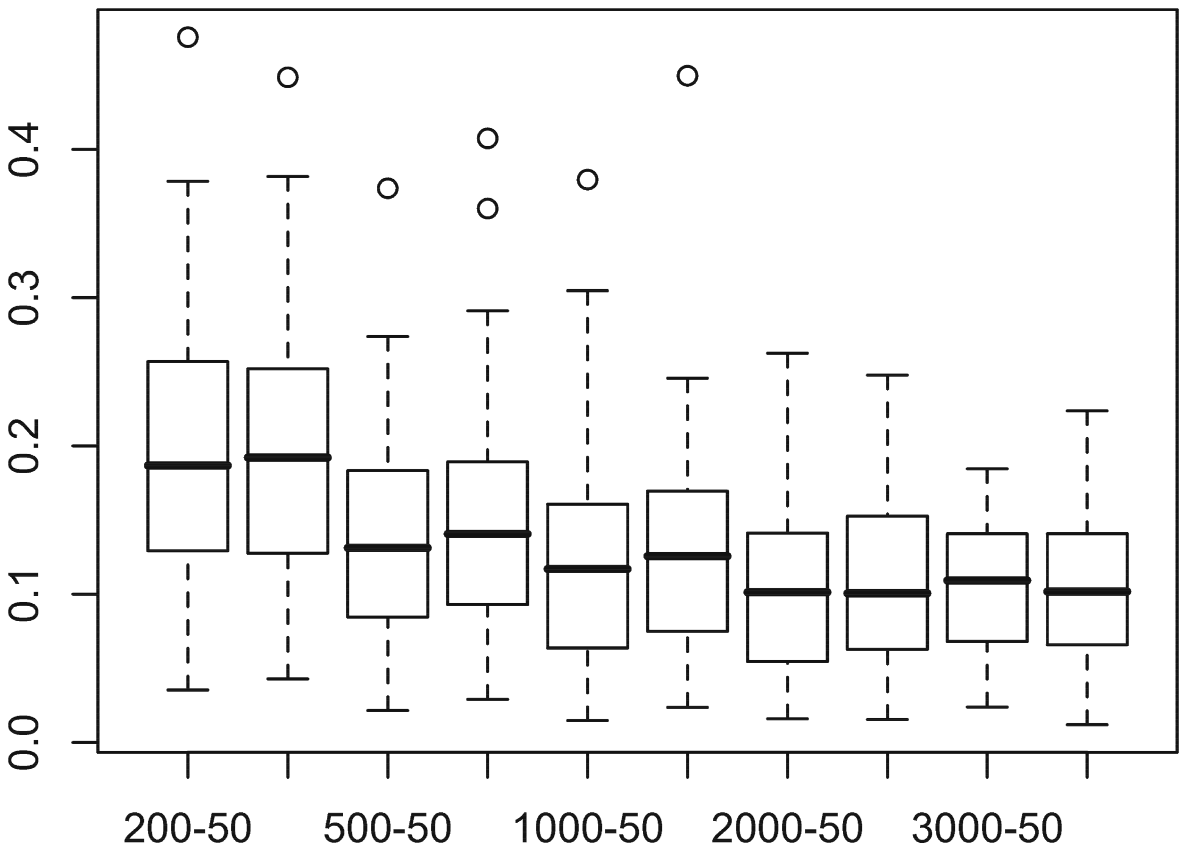}
\end{tabular}
\caption{Boxplots of 100 $L^1$-distances $\int |\hat \Psi(z) - \Psi(\mu + \tau
  z) | dz$ calculated from samples of size $n=10^4$ with Fr\'echet margins of
  parameter $1/\gamma$ and Frank's copula with parameter $\theta$, for
  $\theta=0.3$ and $\gamma=0.1$ (left side) or $\gamma=0.4$ (right side). The
  10 boxplots correspond to different choices of thresholds $k$ and $k_\gamma$:
  from left to right, $(k, k_\gamma)$ has the values: $(200,50)$, $(200,200)$,
  $(500,50)$, $(500,500)$, $(1000,50)$, $(1000,500)$, $(2000,50)$,
  $(2000,500)$, $(3000,50)$, $(3000,500)$.  }
\label{fig:distL1_10e4_th03_gam01}
\end{figure}
A common feature of both plots of Figure~\ref{fig:distL1_10e4_th03_gam01} is
that the results do not depend much on the choice of the threshold $k_\gamma$,
but are a bit more sensitive to the choice of the second threshold
$k$. Besides, the results show that the estimators provided in Section~3
perform well in a context that is rather general.  Finally, these performances
are better when the parameter $\gamma$ is smaller.

\subsection{Goodness-of-fit test for the distribution $\Psi$}
\label{subsec:gof}
Assume that the hypotheses of Section~\ref{sec:product} hold, so that
the nonparametric estimation procedure described in
Section~\ref{subsec:nonparam-estim} can be used.  Three types of
distributions are considered, each of them restricted to the positive
quadrant for convenience.  These distributions
are: \\
{\bf (a)} the elliptical distribution with radial survival function
$P(R>t)=\mathrm{e}^{-t}$, and Pearson correlation coefficient
$\rho=0.5$ ; \\
{\bf (b)} the distribution with radial representation
$R(\cos[(\pi/2+\arcsin \rho)T - \arcsin \rho],\sin[(\pi/2+\arcsin
\rho)T])$, where $P(R>t)=\mathrm{e}^{-t^2/2}$, $T$ has a non uniform
concave density function $f_T(t)=4/\left\{
  \pi+\pi(2t-1)^2\right\}$ on $[0,1]$, and $\rho=0.5$;\\
{\bf (c)} the distribution with radial representation
$R(\cos[(\pi/2+\arcsin \rho)T - \arcsin \rho],\sin[(\pi/2+\arcsin
\rho)T])$, where $P(R>t)=\mathrm{e}^{-t^2/2}$, $T$ has a non uniform
convex density function $f_T(t)=2-4/ \left\{ \pi(1+(2t-1)^2 \right\}$
on $[0,1]$, and $\rho=0.5$.

Case {\bf (a)} is an example of the standard elliptical case, for which
estimation results already exist (see
\cite{abdous:fougeres:ghoudi:soulier:2008}), whereas {\bf (b)} and {\bf (c)}
illustrate the situation where the density level lines are ``asymptotically
elliptic'' (see \cite{fougeres:soulier:2010}).  In these three cases, $\Psi$ is
the Normal distribution function (denoted by~$\Phi$), and Assumption 6 is
fulfilled with $\beta=2$.  Figure~\ref{fig:psiplot} illustrates the estimation
of $\Psi$ via the nonparametric estimator $\check \Psi$ defined by
(\ref{eq:NP-Psi}) for one sample ($n=1000, k=100$) of distribution {\bf (b)}.

\begin{figure}[h]
  \centering
  \includegraphics[width=12cm,height=5cm]{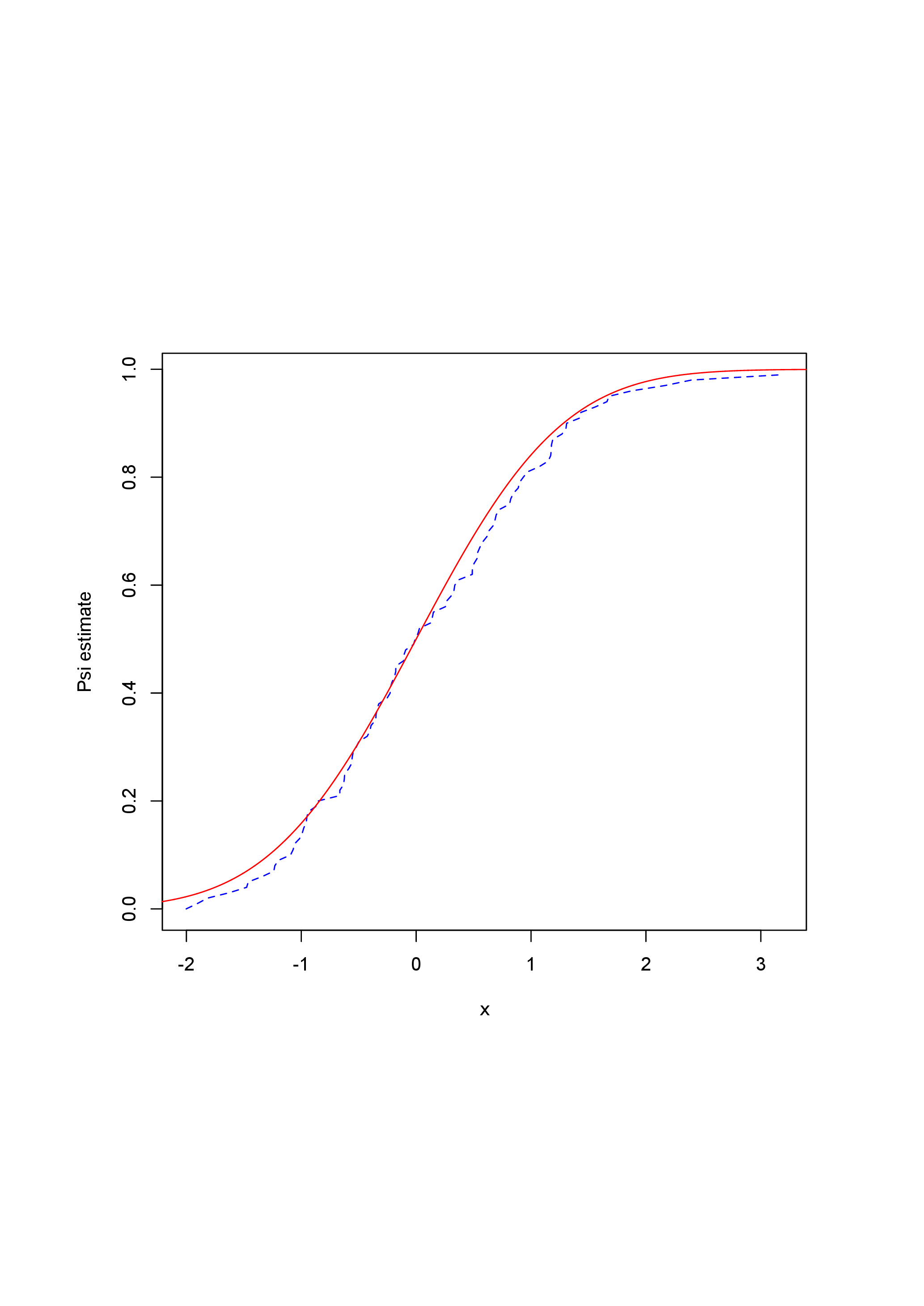}
  \caption{ Estimation of $\Psi$ via the nonparametric estimator $\check \Psi$
    for one sample ($n=1000, k=100$) of distribution {\bf (b)}.}
  \label{fig:psiplot}
\end{figure}

The test statistic $T_{KS}$ of the Kolmogorov-Smirnov goodness-of-fit test
performed here is defined by
\begin{gather}
  \label{eq:stat-KS-norm}
  T_{KS} = \sup_{y \in {\mathbb R}} \sqrt{k} | \check \Psi(y) -
  \Phi(y) | \; .
\end{gather}
As shown in Section~\ref{subsec:nonparam-estim}, $T_{KS}$ has
asymptotically the same distribution as the random variable $Z$
defined in~(\ref{eq:proclim}). Quantiles of this distribution have
been obtained numerically and are listed in
Table~\ref{table:quant-theo}.
\begin{table}[ht]
\begin{center}
\caption{\em Quantiles $q_\alpha$ of order $1-\alpha$
  of $Z$.}
\label{table:quant-theo}
\begin{tabular}{lcccccc}
&&&&&&\\
$\alpha$  & 0.01 & 0.05 & 0.10 & 0.15 & 0.20 & 0.25 \\
\hline
$q_{\alpha}$ &  1.598 & 1.297 & 1.174 &  1.076 & 1.029 & 0.980 \\
\end{tabular}
\end{center}
\end{table}

We have compared these theoretical levels to the empirical levels obtained by
simulation. In the three cases {\bf (a)} to {\bf (c)}, 1000 samples of
size~$n=10^3$, $10^4$ and $10^5$, are simulated. The $k$ observations having
the largest first component are kept, for three different values of $k$, and
the nonparametric estimate $\check \Psi$ given in (\ref{eq:NP-Psi}) is computed
with this reduced sample.  The observed values of the test statistic $T_{KS}$
are compared to the quantiles listed in Table~\ref{table:quant-theo}. For
brevity, we present only the results corresponding to the two theoretical
levels $\alpha=(0.05,0.1)$. These empirical levels are shown in
Table~\ref{table:emp-level}.

\begin{table}[h]
  \caption{\em
    Empirical levels $(\hat \alpha_{0.05}, \hat \alpha_{0.1}) $
    associated to theoretical levels $(0.05, 0.1)$
    for the goodness-of-fit test with statistic $T_{KS}$.
    The original sample size is denoted by $n$,
    and the number of observations used for the estimation is denoted by 
    $k$. Notation
    {\bf (a)}--{\bf (c)} refers to the three bivariate distributions listed 
    above. The boldface characters point out the best result in each case. }
\label{table:emp-level}
$$\begin{array}{ccccc}
n & k & {\mbox{\bf (a)}} & {\mbox{\bf (b)}} & {\mbox{\bf (c)}} \\
\hline
& 50 &  {{\bf (0.053, 0.095)}} & (0.031, 0.066) & (0.027,0.050)\\
1000 & 100 &  (0.140, 0.231) &  {{\bf (0.055, 0.102)}} &  {{\bf(0.04, 
0.085)}} \\
& 150 &  (0.327, 0.453 )                           & (0.071, 0.147 ) & 
(0.077, 0.153) \\
\hline
& 50 & {(0.059, 0.095 )} & ( 0.03, 0.061 ) & (0.028 , 0.045)\\
10000& 100 & {{\bf (0.052, 0.099)}} & (0.038,0.07) & {\bf (0.038, 0.088)} \\
 & 200  &  {( 0.101, 0.183)} & {\bf (0.054,0.096)} & {(0.065, 0.125)}\\
\hline
& 100 & {\bf (0.051, 0.082)} & (0.037,0.075) &  {\bf (0.044, 0.071)}\\
100000 & 200  & (0.080, 0.133) & {(0.041,0.087)} &  {( 0.0795, 0.128)}\\
& 500  &  {(0.140, 0.257)} & {\bf (0.05,0.103)} & ( 0.20, 0.298)\\
\hline
\end{array}
$$
\end{table}

A common feature for the three distributions is that the results are rather
sensitive to the reduced number of observations $k$. However, the value of $k$
leading to the best adequation between empirical and theoretical levels is
rather stable in most cases studied ($k=100$ in two thirds of the cases).

\subsection{Semiparametric estimation of the conditional quantile function}
\label{subsec:sp-estim}
Assume that
Assumptions~\ref{hypo:exces},~\ref{hypo:varrap-produit},~\ref{hypo:weibull} and
equation~(\ref{eq:simplif}) hold and that the limiting distribution $\Psi$ is
the standard Gaussian distribution $\Phi$.  The small sample behavior of the
semiparametric estimator $\hat y_n(p)$ of the quantile function
$\theta^\leftarrow(x_n,p)$ defined by Equation~(\ref{eq:ynhat}) is illustrated
in Figure~\ref{fig:quantile-estim} for the three distributions presented in
Section~\ref{subsec:gof}.  In each case, 100 samples of size 10000 are
simulated. A proportion of 1\% of the observations is used, which are the 100
observations with largest first component.  For each sample, the conditional
quantile function $\theta^\leftarrow(x,p)$ is estimated for two values of $x$
corresponding to the theoretical $X$-quantiles of order $1-\epsilon$, where
$\epsilon=10^{-4}$ and $\epsilon=10^{-5}$.  Figure~\ref{fig:quantile-estim}
summarizes the quality of these estimations by showing the median, and the
2.5\%- and 97.5\%-quantiles of $\hat y_n(p)$ for the two fixed values of $x$
specified above.

The estimation results are globally good, and the best ones are obtained for
cases {\bf (a)} and {\bf (c)}, see rows 1 and 3 of
Figure~\ref{fig:quantile-estim}.  Besides, one can observe a slight improvement
as the conditioning event becomes more extreme.

These empirical interval confidence compare well with those obtained
by applying the central limit theorem of
Corollary~\ref{coro:quantile}. We do not show them on
Figure~\ref{fig:quantile-estim} for the sake of clarity.

\begin{figure}
  \centering
\includegraphics[width=14cm]{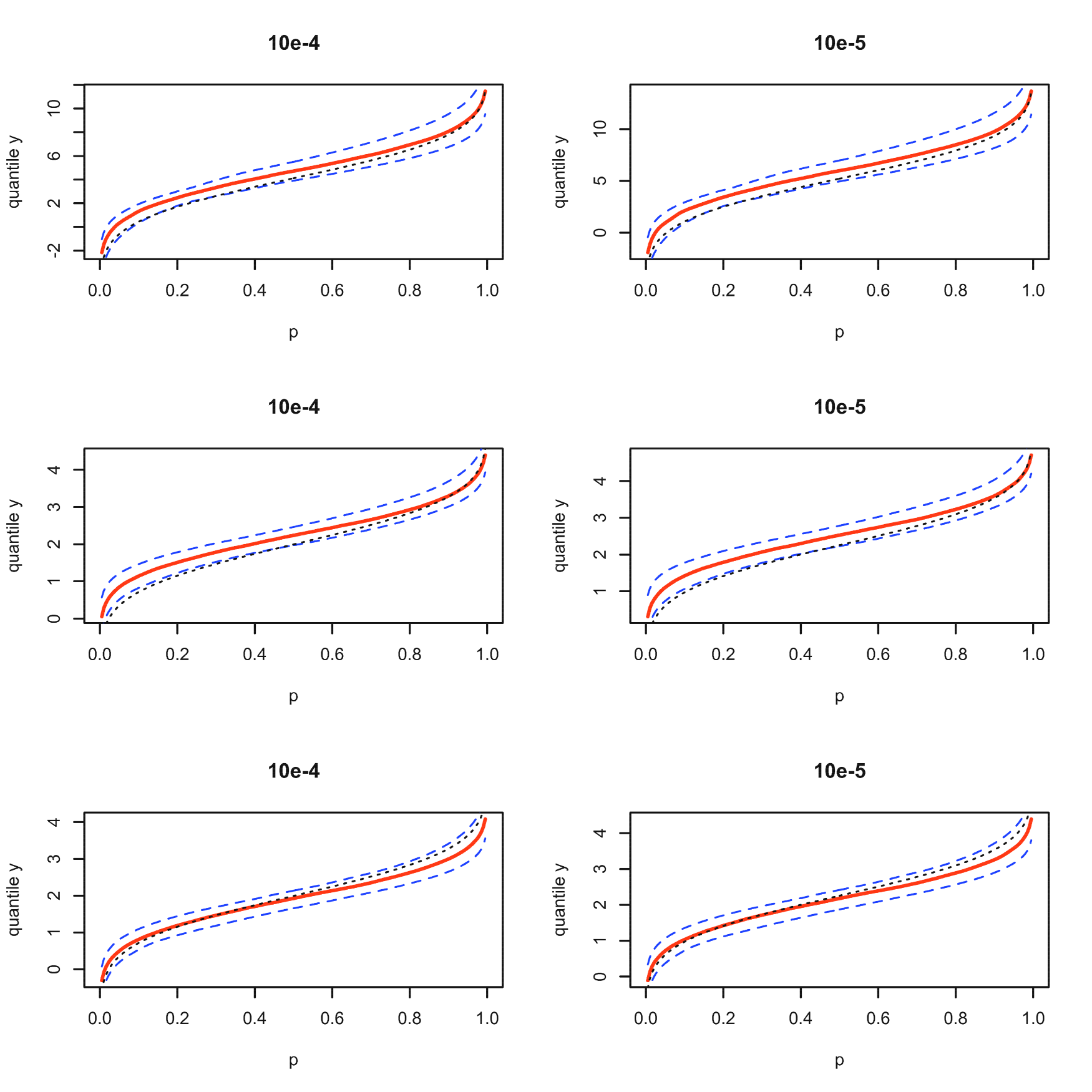}
  \caption{
    Median (solid line), 2.5\%- and 97.5\%-quantiles (dashed lines) of
    the estimated conditional quantile function $\hat
    y=\theta^\leftarrow(x,p)$ defined in~(\ref{eq:ynhat}) and
    theoretical conditional quantile function $y$ (dotted line) as a
    function of the probability $p \in (0,1)$.  Each row (from 1 to 3)
    corresponds to a distribution (from {\bf (a)} to {\bf (c)}) as
    described in Section~\ref{subsec:gof}.  Each column refers to a
    different value of $x$, respectively corresponding to the
    theoretical $X$-quantiles of order $1-\epsilon$, where
    $\epsilon=10^{-4}$ and $p=10^{-5}$.  }
  \label{fig:quantile-estim}
\end{figure}

\section{Data analysis}
\label{sec:data}
To illustrate the use of the new procedures, and more specifically the
Kol\-mo\-go\-rov-Smirnov goodness-of-fit test proposed in
Section~\ref{subsec:nonparam-estim}, the hypothesis of $\Psi=\Phi$,
where $\Phi$ is the standard Gaussian cdf, is tested using the series
of monthly returns for the 3M stock and the Dow Jones Industrial
Average from January 1970 to January 2008 ($n=457$ values). These data
were used by \cite{levy:duchin:2004} and revisited by
\cite{abdous:fougeres:ghoudi:soulier:2008}. In the latter paper, the
hypothesis of bivariate ellipticity was accepted through a test of
elliptical symmetry proposed by \cite{huffer:park:2007} and the
contagion from the Dow Jones to the 3M stock was tested.  As shown in
\cite{abdous:fougeres:ghoudi:2005}, ellipticity implies that
Condition~(\ref{eq:loi-limite}) holds and that the limiting
distribution is the Gaussian law. The present procedure allows to test
for the Gaussian conditional limit law without assuming ellipticity,
but the weaker assumption~(\ref{eq:loi-limite}).  The observed values
of the test statistic $T_{KS}$ defined by (\ref{eq:stat-KS-norm}) in
terms of different choices of threshold $k$ (or equivalently in terms
of the proportion $r$ of observations used, $k=nr$) are summarized in
Table~\ref{table:pval-data}. According to
Table~\ref{table:quant-theo}, all these observed values correspond to
a $p$-value greater than 0.25, which leads to accept the hypothesis
$\Psi=\Phi$.

\begin{table}[h]
\begin{center}
\caption{\em Observed values $t_{KS}$ of the test statistic $T_{KS}$ 
  defined by (\ref{eq:stat-KS-norm}) in terms of the proportion $r$ or
  number~$k$ of observations used.}
\label{table:pval-data}
\begin{tabular}{lcccc}
&&&&\\
$r$  & 0.05 &  0.10 & 0.15 & 0.20 \\
\hline
$k$  & 22 & 45 & 68 & 91  \\
\hline
$t_{KS}$ & 0.842   & 0.847 & 0.777 &  0.948  \\
\end{tabular}
\end{center}
\end{table}

\section{Proofs}
\label{sec:preuves}

\begin{proof} [Proof of Proposition~\ref{prop:moments-tildenun}]
  By Proposition~\ref{prop:conv-tildenun}, the weak convergence of
  $\tilde\nu_n$ to $\nu$ implies that for any compact set $K$ of
  $(-1/\gamma,\infty) \times (-\infty,\infty)$ such that $\nu(\partial
  K)=0$ and any function $h$, it holds that
  \begin{gather*}
    \lim_{n\to\infty} \iint_K h(x,y) \tilde\nu_n(\mathrm d x, \mathrm d y) = \iint_K
    h(x,y) \nu(\mathrm d x, \mathrm d y) \ \mbox{ in probability.} 
  \end{gather*}
  For $\epsilon,M>0$, $\epsilon<1/\gamma$, define $K =
  [-\epsilon,M]\times[-M,M]$ and $K^{c} = [-\epsilon,\infty) \times
  (-\infty,\infty)\setminus K$. Let $h$ be a nonnegative function on
  $[-\epsilon,\infty) \times (-\infty,\infty)$ such that $h(x,y) \leq
  C (|x|\vee1)^{q^*-1} (|y|\vee1)^{p^*-1}$. We must prove that
\begin{align} \label{eq:petite-tension}
  \limsup_{M\to\infty} \lim_{n\to\infty} \iint_{K^c} h(x,y) \tilde
  \nu_n(\mathrm d x, \mathrm d y) = 0 \; ,
\end{align}
in probability.
Since 
\begin{align*}
  \mathbb{E}\left[ \iint_{K^c} h(x,y) \tilde \nu_n (\mathrm d x,\mathrm d y) \right] = \iint_{K^c}
  h(x,y) \nu_{n/k}(x,y) \; ,
\end{align*}
Assumption~\ref{hypo:moments} implies that 
\begin{multline*}
  \lim_{M\to\infty} \limsup_{n\to\infty} \mathbb{E}\left[ \iint_{K^c} h(x,y)
  \tilde  \nu_n  (\mathrm d x,\mathrm d y) \right] \\
  = \lim_{M\to\infty} \limsup_{n\to\infty} \iint_{K^c} h(x,y)
  \nu_{n/k}(\mathrm d x,\mathrm d y) 
  = \lim_{M\to\infty} \iint_{K^c} h(x,y) \nu(\mathrm d x,\mathrm d y) \;  .
\end{multline*}
This yields~(\ref{eq:petite-tension}) and concludes the proof of
Proposition~\ref{prop:moments-tildenun}.
\end{proof}

\begin{proof}[Proof of Proposition~\ref{prop:coco}]
  Assumption~\ref{sec:hypo} and \cite[Proposition 5.3]{resnick:1986}
  imply that the sequence of point processes
  \begin{align*}
    \sum_{k=1}^n \delta_{\frac{X_k - b(n)}{\psi\circ b(n)},\frac{Y_k - m\circ b(n)}{a\circ b(n)}}
  \end{align*}
  converges weakly to a Poisson point process on
  $(-1/\gamma,\infty]\times(-\infty,\infty)$ with intensity measure $\nu$. This
  implies that the vector $({(X_{(n:n)} - b(n))}/{\psi\circ b(n)},({Y_{[n:n]} -
    m\circ b(n)})/{a\circ b(n)})$ converges weakly to the distribution $F$ defined
  in~(\ref{eq:loi-limite}). 

\end{proof}
\begin{proof}[Proof of Proposition~\ref{prop:estim-m}]
  Write
  \begin{align*}
    \frac{ \hat m(X_{(n:n-k)}) - m\circ b(n/k)}{a\circ b(n/k)} =
    \frac{S_n}{T_n}  \; ,
  \end{align*}
with
\begin{align*}
  S_n & = \frac1k \sum_{i=1}^k \frac{Y_{[n:n-i+1]} - m\circ b(n/k) }
  {a \circ b(n/k)} \; \frac{X_{(n:n-i+1)} - X_{(n:n-k)}}
  {\psi \circ b(n/k)} \; ,  \\
  T_n & = \frac1k \sum_{i=1}^k \frac{X_{(n:n-i+1)} - X_{(n:n-k)}}{\psi
    \circ b(n/k)} \; .
\end{align*}
We have already seen that $T_n$ converges weakly to~$1/(1 - \gamma)$.
Recall that we have defined
\begin{gather*}
  \tilde x_n = \frac{X_{(n:n-k)} - b(n/k)}{\psi \circ b(n/k)} \; .
\end{gather*}
By definition of $\tilde \nu_n$, we have, (with $x_+=\sup(x,0)$ for
any real number $x$)
\begin{align}
  S_n & = \frac1k \sum_{i=1}^n \frac{Y_i - m \circ b(n/k) }
  {a(X_{(n:n-k)})} \; \left\{ \frac{X_i - b(n/k) } {\psi\circ b(n/k)}
    - \tilde x_n \right\}_+ = \int_{\tilde x_n}^\infty
  \int_{-\infty}^\infty (x-\tilde x_n) y
  \,  \tilde \nu_n(\mathrm d x,\mathrm d y)  \nonumber \\
  & = \int_{0}^\infty \int_{-\infty}^\infty x y \, \tilde \nu_n(\d
  x,\mathrm d y) - \int_0^{\tilde x_n} \int_{-\infty}^\infty x y \, \tilde
  \nu_n(\mathrm d x,\mathrm d y) - \tilde x_n \int_{\tilde x_n}^\infty
  \int_{-\infty}^\infty y \, \tilde \nu_n(\mathrm d x,\mathrm d y) \; .
  \label{eq:tout-avoue}
\end{align}
By Proposition~\ref{prop:moments-tildenun}, the first term
in~(\ref{eq:tout-avoue}) converges to $\mu/(1-\gamma)$.  Under
Assumption~\ref{hypo:exces}, it is well known that $\tilde x_n =
o_P(1)$.  Cf.  \citet[Theorem~2.2.1]{dehaan:ferreira:2006}. This and
Assumption~\ref{hypo:moments} imply that the last two terms
in~(\ref{eq:tout-avoue}) are $o_P(1)$.  Thus $S_n$ converges weakly to
$\mu/(1-\gamma)$ by Proposition~\ref{prop:moments-tildenun}.  If
$m(x)=\rho x$, then $\hat \rho - \rho \sim X_{(n:n-k)}^{-1}
a(X_{(n:n-k)}) \mu$ which converges to 0 if $a(x)=o(x)$ or if $\mu=0$
and $a(x)=O(x)$.
\end{proof}

\begin{proof}[Proof of Proposition~\ref{prop:estim-a}]
  We show that $\hat a^2(X_{(n:n-k)})/ a^2 \circ b(n/k)$ converges
  weakly to~1. Recall that 
    $\xi_n = \{\hat m(X_{(n:n-k)}) - m \circ b (n/k)\}/a \circ b (n/k)$. 
  By Proposition~\ref{prop:estim-m},  $\xi_n = o_P(1)$, and noting
  that $\tilde\nu_n\{[\tilde x_n,\infty]\times[-\infty, \infty]\}=1$,
  where $\tilde\nu_n$ and $\tilde x_n$ are respectively defined by
  (\ref{eq:def-tilde-nu_n}) and (\ref{eq:def-xin}), we have
  \begin{align*}
    \frac{\hat a^2(X_{(n:n-k)})}{a^2\circ b(n/k)} & = \frac1k
    \sum_{i=1}^n \left\{\frac{ Y_i-m\circ b(n/k)}{a\circ b(n/k)} -
      \xi_n \right\}^2\mathbf{1}_{\{ \frac{X_i-b(n/k)}{\psi\circ b(n/k)}
      \geq \tilde x_n\}} 
    \\
    & = \int_{\tilde x_n}^\infty \int_{-\infty}^\infty
    (y-\xi_n)^2 \,
    \tilde\nu_n(\mathrm d x,\mathrm d y)  \\
    & = \int_{\tilde x_n}^\infty \int_{-\infty}^\infty
    y^2 \, \tilde\nu_n(\mathrm d x,\mathrm d y) - 2 \xi_n \int_{\tilde x_n}^\infty
    \int_{-\infty}^\infty y \, \tilde\nu_n(\mathrm d x,\mathrm d y) + \xi_n^2    \\
    & = \int_{0}^\infty \int_{-\infty}^\infty y^2 \, \tilde\nu_n(\d
    x,\mathrm d y) + o_P(1) \; . 
  \end{align*}
  Thus $\hat a(X_{(n:n-k)})/a \circ b(n/k)$ converges weakly to~1 by
  Proposition~\ref{prop:moments-tildenun} and equation
  (\ref{eq:standardisation-variance}).
\end{proof}

\begin{proof}[Proof of Proposition~\ref{prop:clt-fidi}]
  Note first that assumption~(\ref{eq:condition-second-ordre-k}) implies that
  $k$ is an intermediate sequence, i.e. $n/k\to0$, since $\lim_{t\to\infty}
  c\circ b(t) = 0$.  We start by proving the convergence of the finite
  dimensional distributions of $W_n$. Denote $G_n(x,y)=
  \nu_n((x,\infty)\times(-\infty,y])$, $G(x,y) =
  \nu((x,\infty)\times(-\infty,y])$, $\tilde X_i = \{X_i - b(n/k)\}/\psi\circ
  b(n/k)$, $\tilde Y_i = \{Y_i - m \circ b(n/k)\}/a\circ b(n/k)$ and
  \begin{align*}
    \xi_{n,i}(x,y) & = k^{-1/2} \{\mathbf{1}_{\{\tilde X_i > x \, , \;
      \tilde
      Y_i \leq y\}} - \mathbb{P}(\tilde X_i>x \;, \tilde Y_i\leq y)\} \\
    & = k^{-1/2} \{\mathbf{1}_{\{\tilde X_i > x \, , \; \tilde Y_i \leq y\}}
    - kn^{-1} G_{n/k}(x,y)\} \; .
  \end{align*}
  Then for each $n$, the random variables $\xi_{n,i}$, $1 \leq i \leq
  n$ are i.i.d., 
\begin{align*}
  \mathrm{cov}(\xi_{n,i}(x,y),\xi_{n,i}(x',y')) = \frac1n G_{n/k}(x \vee x',y
  \wedge y') - \frac{k}{n^2} G_{n/k}(x,y ) G_{n/k}(x',y' ) \; ,
\end{align*}
and
  \begin{align*}
    W_n(x,y) = \sum_{i=1}^n \xi_{n,i} (x,y) + k^{1/2} \{G_{n/k}(x,y) -
    G(x,y)\} \; .
  \end{align*}
  Assumption~\ref{hypo:second-ordre-modifie}
  and~(\ref{eq:condition-second-ordre-k}) imply that $k^{1/2}(G_{n/k}-G)$
  converges to zero locally uniformly.  The Lindeberg central limit
  theorem (cf. \cite{araujo:gine:1980})
  and~(\ref{eq:condition-second-ordre-k}) yield the convergence of
  finite dimensional distributions of $\sum_{i=1}^n \xi_{n,i} (x,y)$
  to the Gaussian process with covariance defined
  by~(\ref{eq:cov-W}).  

  Let $\mathcal K$ be a compact set of $(-1/\gamma,\infty) \times
  (-\infty,\infty)$.  The tightness of the sequence of processes
  $\{\sum_{i=1}^n \xi_{n,i}(x,y),(x,y)\in \mathcal K\}$ is obtained by applying
  \cite[Example 2.11.8]{vandervaart:wellner:1996} with $c_{n,i} = 1/\sqrt k$,
  $P_{n,i} = P_n = \mathbb{P}((\tilde X_i,\tilde Y_i)\in \cdot)$ and the set of
  functions $\mathcal F$ is the set of indicators
  $\mathbf{1}_{\{(x,\infty)\times(-\infty,y]\}}$ for $(x,y)\in\mathcal K$. The
  conditions of Example 2.11.8 are satisfied, since $\max_{1\le i\leq
    n}|c_{n,i}|\to 0$ trivially, and
  \begin{align*}
    \sum_{i=1}^n c_{n,i}^2P_{n,i} = \frac nk P_n \to F \; , 
  \end{align*}
  by Assumption~\ref{hypo:exces} and Proposition~\ref{prop:conv-tildenun}.
  Finally, the class $\mathcal F$ satisfies the uniform entropy condition, as
  shown in \cite[Example 2.5.4]{vandervaart:wellner:1996}.

  This proves the convergence of the sequence of processes$\sum_{i=1}^n
  \xi_{n,i}$ to $W$ uniformly on compact sets of
  $(-1/\gamma,\infty)\times(-\infty),\infty)$.

  We now prove the second part of Proposition~\ref{prop:clt-fidi}.  Let $h$ be
  a $C^\infty$ function with compact support in
  $(-1/\gamma,\infty)\times(-\infty,\infty)$.  The weak convergence of $W_n$ in
  $\mathcal D((-1/\gamma,\infty)\times(-\infty,\infty))$ implies that $\iint
  h(x,y) W_n(x,y) \, \mathrm d x \, \mathrm d y$ converges weakly to $\iint h(x,y) W(x,y) \,
  \mathrm d x \, \mathrm d y$.  Thus, by integration by parts, it also holds that $\iint
  h(x,y) W_n(\mathrm d x,\mathrm d y)$ converges weakly to $\iint h(x,y) W(\mathrm d x,\mathrm d y)$.  Let
  $\epsilon\in(0,1/\gamma)$ and define
  $A=[-\epsilon,\infty)\times(-\infty,\infty)$. Let $g$ be a measurable
  function defined on $A$ such that $|g(x,y)|^2 \leq
  C(|x|\vee1)^{p^\dag}(|y|\vee1)^{q^\dag}$.  Then, for all $\epsilon>0$, there
  exists a $C^\infty$ function $h$ with compact support in $A$ such that
  $\int_A (g-h)^2 \, \mathrm d \nu \leq \epsilon$.  Then,
  \begin{align*}
    \int_A g \,  \mathrm d \tilde \mu_n = \int_A h \, \mathrm d \tilde  \mu_n + \int_A (g-h)
    \, \mathrm d \tilde \mu_n \; .
  \end{align*}
  The first term in the right hand side converges weakly to $W(h)$ and
  we prove now that the second one converges in probability to 0.
  Denote $u=g-h$ and 
\begin{gather*}
   \mu_n = k^{1/2} \{\nu_{n/k}-\nu\} \; .
\end{gather*}
Then,
  \begin{align*}
    \int_A u \, \mathrm d \tilde \mu_n = k^{-1/2} \sum_{i=1}^n \{
    u(\tilde X_i,\tilde Y_i) - \mathbb{E}[u(\tilde X_i,\tilde Y_i)]\} +
    \int_A u \, \mathrm d \mu_n \; .
  \end{align*}
  By definition, for any function $v$, $\mathbb{E}[v(\tilde X_1)] = kn^{-1}
  \int v \, \mathrm d\nu_{n/k}$, thus
\begin{align*}
  \mathbb{E} \left[ \left( \int_A u \, \mathrm d \tilde \mu_n \right)^2 \right]
  \leq \int_A u^2 \, \mathrm d \nu_{n/k} + \left\{ \int_A u \, \mathrm d \mu_n
  \right \}^2 \; .
\end{align*}
By assumption on $g$, and since $h$ has compact support, it also holds
that $u^2(x,y) \leq C (|x|\vee1)^{p^\dag}(|y|\vee1)^{q^\dag}$. Thus,
by Assumption~\ref{hypo:second-ordre-modifie}
and~(\ref{eq:condition-second-ordre-k}), it holds that
$\lim_{n\to\infty} \int_A u \, \mathrm d \mu_n = 0$ and $\lim_{n\to\infty}
\int_A u^2 \, \mathrm d \nu_n = \int_A u^2 \, \mathrm d \nu$. Thus
\begin{gather*}
  \limsup_{n\to\infty} \mathbb{E} \left[ \left( \int_A u \, \mathrm d \tilde \mu_n
    \right)^2 \right] \leq \int_A u^2 \, \mathrm d \nu \leq \epsilon \; .
\end{gather*}
Taking into account that $\mathrm{var}(W(g)-W(h)) = \mathrm{var}(W(g-h))=
\int_A (g-h)^2 \, \mathrm d \nu \leq \epsilon$, we conclude that $W_n(g)$ converges
weakly to $W(g)$ and that $\mathbb{E}[W^2(g)] = \int g^2 \mathrm d \nu$. The joint
convergence of $W_n(g_1),\dots,W_n(g_k)$ is obtained by the Cramer-Wold device,
and by linearity of $W_n$ and $W$, this is reduced to the one-dimensional
convergence.
\end{proof}

\begin{proof}[Proof of Corollary~\ref{coro:tilde}]
  We prove separately the claimed limit distributions. The joint
  convergence is obvious.
  We start with $\tilde x_n$, defined in~(\ref{eq:def-xin}).
  Denote $\mathbb{G}_n(x) = \tilde
  \nu_n((x,\infty)\times(-\infty,+\infty))$.  By
  Proposition~\ref{prop:clt-fidi}, $k^{1/2}(\mathbb G_n - \bar
  P_\gamma)$ converges weakly in $\mathcal D$ to the process $B \circ
  \bar P_\gamma$, where $B$ is a standard Brownian motion on $[0,1]$.
  By Vervaat's Lemma \citep[Lemma~A.0.2]{dehaan:ferreira:2006},
  $k^{1/2}\{\mathbb G_n^\leftarrow - \bar P_\gamma^\leftarrow\}$
  jointly converges weakly in $\mathcal{D}$ to $-(\bar
  P_\gamma^\leftarrow)' B$.  Since $\mathbb G_n^\leftarrow(1) = \tilde
  x_n$, $\bar P_\gamma^\leftarrow(1)=0$ and $(\bar
  P_\gamma^\leftarrow)'(1) = -1$, we get the claimed limit
  distribution for $k^{1/2}\tilde x_n$.

We now consider $\xi_n$, defined in~(\ref{eq:def-xin}). By definition,
  \begin{align*}
    \xi_n & = \frac{\sum_{i=1}^k \{X_{(n:n-i+1)} - X_{(n:n-k)}\}
      \{Y_{[n:n-i+1]} - m \circ b(n/k)\}}{k \psi \circ
      b(n/k) a\circ b(n/k)}    \\
    & \hspace*{5cm} \div \ \frac{\sum_{i=1}^k \{X_{(n:n-i+1)} -
      X_{(n:n-k)\}} }{k \psi \circ b(n/k)}    \\
    & = \frac{ \int_{\tilde x_n}^\infty \int_{-\infty}^\infty
      (x-\tilde x_n) y \tilde \nu_n(\mathrm d x, \mathrm d y)} { \int_{\tilde
        x_n}^\infty \int_{-\infty}^\infty (x - \tilde x_n) \tilde
      \nu_n(\mathrm d x, \mathrm d y)} \; .
  \end{align*}
  Since $\mu=0$ by assumption, we obtain
  \begin{gather*}
    k^{1/2} \xi_n = \frac{ \int_{\tilde x_n}^\infty
      \int_{-\infty}^\infty (x-\tilde x_n) y \tilde \mu_n(\mathrm d x, \mathrm d y)}
    { \int_{\tilde x_n}^\infty \int_{-\infty}^\infty (x - \tilde x_n)
      \tilde \nu_n(\mathrm d x, \mathrm d y)} \; .
  \end{gather*}
  Since $\tilde x_n = O_P(k^{-1/2})$, it is easily seen that
  \begin{gather*}
    k^{1/2} \xi_n = \frac{ \int_{0}^\infty \int_{-\infty}^\infty x y
      \tilde \mu_n(\mathrm d x, \mathrm d y) + o_P(1)} { \int_{0}^\infty
      \int_{-\infty}^\infty x \tilde \nu_n(\mathrm d x, \mathrm d y) + o_P(1)} \; .
  \end{gather*}
  Applying Propositions~\ref{prop:moments-tildenun}
  and~\ref{prop:clt-fidi}, we obtain that $k^{1/2}\xi_n$ converges
  weakly to $(1-\gamma)W(g_{1,1})$.  Consider now $\hat
  a(X_{(n:n-k)})$. As in the proof of Proposition~\ref{prop:estim-a},
  we write
  \begin{align*}
    \frac {\hat a^2 (X_{(n:n-k)})}{a^2\circ (n/k)} & = \int_{ \tilde
      x_n}^\infty \int_{-\infty}^\infty y^2 \tilde \nu_n(\mathrm d x,\mathrm d y) -
    2 \xi_n \int_{\tilde x_n}^\infty \int_{-\infty}^\infty y \tilde
    \nu_n(\mathrm d x,\mathrm d y) + \xi_n^2 \; ,
  \end{align*}
  and since $\tilde x_n = O_P(k^{-1/2})$ and $\xi_n=O_P(k^{-1/2})$, we
  get
\begin{align*}
  k^{1/2} \left\{ \frac{\hat a^2(X_{(n:n-k)})}{a^2\circ b(n/k)} - 1
  \right\} = \int_0^\infty \int_{-\infty}^\infty y^2 \tilde \mu_n (\d
  x, \mathrm d y) + o_P(1) \; .
\end{align*}
Proposition~\ref{prop:clt-fidi} and the delta method yield that
$k^{1/2} \{{\hat a(X_{(n:n-k)})}/{a\circ b(n/k)} - 1\}$ converges
weakly to $\frac12 W(g_{0,2})$.
\end{proof}

\begin{proof}[Proof of Proposition~\ref{prop:breve-a}]
  The asymptotic normality of $\hat \beta_k$ is proved (under more
  general conditions) in \citet[Corollary~1]{gardes:girard:2006}.
  Consider now $\breve a_{k_1}(x_n)$. 
  By~(\ref{eq:simplif}) and~(\ref{eq:psi-semiparametrique}),
  \begin{align*}
    a(x) = a(X_{(n:n-k_1)})
  \left(\frac{x}{X_{(n:n-k_1)}}\right)^{1-\beta/2} \; ,
  \end{align*}
thus, by~(\ref{eq:def-breve-a}), we obtain
\begin{align*}
  \frac{\breve{a}_{k_1}(x_n)}{a(x_n)} = \frac{\check a_{k_1} (X_{(n:n-k_1)})}
  {a(X_{(n:n-k_1)})} \; X_{(n:n-k_1)}^{(\hat \beta_k-\beta)/2} \;
  x_n^{(\beta-\hat\beta_k)/2} \; .
\end{align*}
Decomposing further, we get 
\begin{align}
  \frac{\breve{a}_{k_1}(x_n)}{a(x_n)} - 1 & = \left\{\frac{\check a_{k_1}
      (X_{(n:n-k_1)})} {a(X_{(n:n-k_1)})} -1 \right\}
  X_{(n:n-k_1)}^{(\hat \beta_k-\beta)/2} \; x_n^{(\beta-\hat\beta_k)/2}
  \label{eq:negligible1}  \\
  & + \left\{ X_{(n:n-k_1)}^{(\hat \beta_k - \beta)/2} -1 \right\} \left\{
    x_n^{(\beta-\hat\beta_k)/2} - 1 \right \} \label{eq:negligible2} \\
  & + X_{(n:n-k_1)}^{(\hat \beta_k - \beta)/2} -1 +
  x_n^{(\beta-\hat\beta_k)/2} - 1 \; .
\end{align}
Since $\hat \beta_k - \beta= O_P(k^{-1/2})$, $\log(x_n) = o(k^{1/2})$
and $k/k_1\to0$, we obtain
\begin{align*}
  x_n^{(\beta-\hat\beta_k)/2} - 1 & \sim (\beta-\hat \beta_k) \log(x_n)/2 \; , \\
  X_{(n:n-k_1)}^{(\beta-\hat\beta_k)/2} - 1 &  \sim (\beta-\hat \beta_k)
  \log(X_{(n:n-k_1)})/2 \sim (\beta-\hat \beta_k) \log(b(n/k_1))/2 \; ,
\end{align*}
where the equivalence relations above hold in probability.  Thus, by
the first part of Proposition~\ref{prop:breve-a}
and~(\ref{eq:logbnsimlogxn}) the product in~(\ref{eq:negligible2}) is
$O_P(k^{-1}\log^2(x_n)) = o_P(k^{-1/2}\log(x_n))$ by~(\ref{eq:klogn}).
By Theorem~\ref{theo:fclt-produit}, $\check
a_{k_1}(X_{(n:n-k_1)})/a(X_{(n:n-k_1)}) - 1 = O_P(k_1^{-1/2})$, thus the
term in the right hand side of~(\ref{eq:negligible1}) is
$O_P(k_1^{-1/2}) = o_P(k^{-1/2}\log(x_n))$ since $k/k_1\to0$.
Altogether, these bounds yields,
\begin{align*}
  \frac{k^{1/2}}{\log(x_n)} \left\{\frac{\breve{a}_{k_1}(x_n)}{a(x_n)} -
    1\right\} = k^{1/2} (\beta - \hat \beta_k) + o_P(1) \; ,
\end{align*}
and the proof follows from the asymptotic normality of $k^{1/2} (\beta
- \hat \beta_k)$.
\end{proof}

\begin{proof}[Proof of Corollary~\ref{coro:quantile}]
  Define $\tilde y_n = \rho x_n + a(x_n) \Psi^{-1}(p)$. Then
\begin{align*}
  \hat y_n - y_n & = \hat y_n - \tilde y_n + \tilde y_n - y_n \\
  & = (\hat \rho_{k_1} -\rho) x_n + (\breve a_{k_1}(x_n) - a(x_n)) \Psi^{-1}(p) +
  \tilde y_n - y_n \; .
\end{align*}
In order to study $\tilde y_n-y_n$, denote $z_n = (y_n - \rho
x_n)/a(x_n)$. Then $\lim_{n\to\infty} z_n = \Psi^{-1}(p)$. Indeed, if
the sequence $z_n$ is unbounded, then it tends to infinity at least
along a subsequence. Choose $z > \Psi^{-1}(p)$. Then, for large
enough~$n$,
\begin{align*}
  p & = \mathbb{P}(Y \leq \rho x_n + a(x_n) z_n \mid X>x_n) \geq \mathbb{P}(Y
  \leq
  \rho x_n + a(x_n) z \mid X>x_n) \\
  & \to \Psi(z) > p \; .
\end{align*}
Thus the sequence $z_n$ is bounded, and if it converges to $z$ (along
a subsequence), it necessarily holds that $\Psi(z) = p$, thus $z_n$
converges to $\Psi^{-1}(p)$. Since we have assumed that $a(x) = o(x)$,
this implies that $y_n \sim \rho x_n$ and 
\begin{gather*}
  \frac{\tilde y_n - y_n}{y_n} \sim \frac{a(x_n)\{\Psi^{-1}(p)-z_n\}}{\rho
    x_n}  \to 0 \; . 
\end{gather*}
Moreover, since $\Psi'\circ\Psi^{-1}(p)>0$, by a first order Taylor
expansion, we have
\begin{align*}
  \Psi^{-1}(p) - z_n = \frac1{\Psi'(\xi_n)} \{ \theta(x_n,y_n) -
  \Psi(z_n) \} \; , 
\end{align*}
where $\xi_n = \Psi^{-1}(p) + u\{z_n - \Psi^{-1}(p)\}$ for some
$u\in(0,1)$.  By Assumption~\ref{hypo:second-ordre-modifie},
$\|\theta(x_n,\rho x_n+a(x_n)\cdot)-\Psi\|_\infty = O(c \circ b(n))$.
Since we have already shown that $z_n$ converges to $\Psi^{-1}(p)$,
$1/\Psi'(\xi_n)$ is bounded for large enough $n$, so $ \Psi^{-1}(p) -
z_n = O(c \circ b(n))$. Thus, by~(\ref{eq:rate-k_1}) (with $c$
instead of $\tilde c$), we get
\begin{align*}
  \frac{k^{1/2}x_n}{\log(x_n)a(x_n)} \frac{\tilde y_n - y_n}{y_n} =
  O\left( \frac{k^{1/2} c \circ b(n)}{\log(x_n)} \right) = o\left(
    \frac{k_1^{1/2} c \circ b(n)}{\log(x_n)} \right) = o(1) \; .
\end{align*}
Next, by definition, and since $y_n \sim \rho x_n$ and
$a(x_n)=o(x_n)$, we have
\begin{align*}
  \frac{\hat y_n - \tilde y_n}{y_n} \sim  \frac{\hat \rho_{k_1} -
    \rho}{\rho} + \frac{a(x_n) \Psi^{-1}(p)}{\rho x_n} \left\{
    \frac{\breve a_{k_1}(x_n)}{a(x_n)} - 1 \right\} \; .
\end{align*}
Thus, 
\begin{align*}
  \frac{k^{1/2}}{\log(x_n)}\frac{x_n}{a(x_n)} \frac{\hat y_n - \tilde
    y_n}{y_n} & \sim \frac{k^{1/2} x_n(\hat \rho_{k_1} - \rho)}{\rho a(x_n)
    \log(x_n) } + \frac{ \Psi^{-1}(p)}{\rho} \frac{k^{1/2}}{\log(x_n)}
  \left \{ \frac{\breve a_{k_1}(x_n)}{a(x_n)} - 1 \right \} \; . 
\end{align*}
The first term in the right-hand side tends to zero by
Theorem~\ref{theo:fclt-produit} and the assumptions on the sequences
$k_1$, $k$ and $x_n$.  The second term converges weakly to a centered
Gaussian law with variance $\{\Psi^{-1}(p)/(\rho\beta)\}^2$ by
Proposition~\ref{prop:breve-a}. In the case $\Psi^{-1}(p)=0$, the main term
is the first one in the right-hand side of the last display, and we
conclude by applying Theorem~\ref{theo:fclt-produit}.
 \end{proof}

\end{document}